\documentclass[12pt]{amsart}
\usepackage{amssymb}
\input{epsf}
\input xy  \xyoption{all}

\newtheorem{fact}{Fact}[section]
\newtheorem{lemma}[fact]{Lemma}
\newtheorem{theorem}[fact]{Theorem}
\newtheorem{defi}[fact]{Definition}
\newtheorem{exa}[fact]{Example}
\newtheorem{rremark}[fact]{Remark}
\newtheorem{proposition}[fact]{Proposition}
\newtheorem{corollary}[fact]{Corollary}
\newenvironment{remark}{\begin{rremark} \rm}{\end{rremark}}
\newenvironment{definition}{\begin{defi} \rm}{\end{defi}}
\newenvironment{example}{\begin{exa} \rm}{\end{exa}}
\newtheorem{conjecture}[fact]{Conjecture}

\def\semmi#1{}

\newcommand{\smx}[4]{\bigl(\begin{smallmatrix}{#1}&{#2}\\{#3}&{#4}\end{smallmatrix}\bigr)}

\DeclareMathOperator{\ts}{Ts}
\DeclareMathOperator{\Tp}{Tp}
\DeclareMathOperator{\tp}{tp}

\DeclareMathOperator{\E}{\mathcal E}

\DeclareMathOperator{\C}{\mathbf C} \DeclareMathOperator{\Z}{\mathbf Z}
\DeclareMathOperator{\N}{\mathbf N}
\DeclareMathOperator{\K}{\mathcal K}

\DeclareMathOperator{\M}{\mathcal M}
\renewcommand{\P}{\mathbf P}
\DeclareMathOperator{\Hom}{Hom}
\def\iso{\cong}
\DeclareMathOperator{\gr}{Gr}       \DeclareMathOperator{\res}{Res}
\DeclareMathOperator{\codim}{codim}  \DeclareMathOperator{\sym}{Sym}
   \DeclareMathOperator{\pt}{pt}
 \DeclareMathOperator{\sing}{Sing}
\DeclareMathOperator{\GL}{GL}
\DeclareMathOperator{\Sym}{Sym}

\DeclareMathOperator\RES{\mathbf {RES}}
\DeclareMathOperator\dis{disc}
\DeclareMathOperator\depth{depth}
\DeclareMathOperator\Asym{ASym}

\begin{document}

\title{Thom series of contact singularities}

\author{L. M. Feh\'er}
\address{Department of Analysis, Eotvos University Budapest, Hungary}
\email{lfeher@renyi.hu}

\author{R. Rim\'anyi}
\address{Department of Mathematics, University of North Carolina at Chapel Hill, USA}
\email{rimanyi@email.unc.edu}

\dedicatory{Dedicated to Jim Damon}

\thanks{\noindent The first author was supported by the Alfred Renyi Mathematical Institute,
OTKA grants 46365. He is presently supported by the Bolyai Janos Scholarship and OTKA grants 72537 and 81203.
The second author is supported by the Marie Curie Fellowship PIEF-GA-2009-235437 and by NSA grant H98230-10-1-0171 .
\\ AMS Subject classification 32S20}

\begin{abstract} Thom polynomials measure how global topology forces singularities.
The power of Thom polynomials predestine them to be a useful tool not only in differential topology, but also in algebraic geometry (enumerative geometry, moduli spaces) and algebraic combinatorics. The main obstacle of their widespread application is that only a few, sporadic Thom polynomials have been known explicitly. In this paper we develop a general method for calculating Thom polynomials of singularities. Along the way, relations with the equivariant geometry of (punctual, local) Hilbert schemes, and with iterated residue identities are revealed.
\end{abstract}

\maketitle

\section{Introduction}

For a holomorphic map $F:N^n\to P^p$ between compact complex
manifolds one can consider the set of points $\eta(F)$ in the source
manifold $N$ where the map has a certain kind of singularity $\eta$. The
Thom polynomial $\Tp(\eta)$ of $\eta$ is a multivariate polynomial
depending only on $\eta$, with the property that the cohomology class
represented by the closure of $\eta(F)$ is equal to the specialization
of $\Tp(\eta)$ at the characteristic classes $c_i(N)$, $F^*(c_i(P))$.
For this statement to hold, the map $F$ must satisfy transversality
conditions. There is an analogous theory for real smooth maps, where one
studies a polynomial of the Stiefel-Whitney classes of $TN$ and $F^*TP$
expressing $[\eta(F)]\in H^*(N;\Z_2)$. These real Thom polynomials can
be calculated from the complex Thom polynomials \cite{borel-haefliger},
hence we restrict our study to the complex case.

We must specify what the singularity $\eta$ means. For the definition of
$\eta(F)$ to make sense, $\eta$ must be a subset of $\E_0(n,p)$, the vector space of holomorphic map germs $(\C^n,0)\rightarrow (\C^p,0)$, invariant under
the action of the holomorphic reparametrization groups of $(\C^n,0)$ and
$(\C^p,0)$. A natural choice for such a subset is obtained by
considering those germs whose local algebras (see definition below) are isomorphic.
Subsets obtained by this way are called contact singularities. In the
language of equivariant cohomology, the Thom polynomial of the first
paragraph is the $\GL(n)\times \GL(p)$-equivariant cohomology class represented by the closure of the contact class $\eta$ in $H_{\GL(n)\times \GL(p)}^*(\E_0(n,p))$.

Thom polynomials have applications in various parts of differential topology, algebraic geometry, and algebraic combinatorics, let us just allude to the simplest case, the celebrated Giambelli-Thom-Porteous
formula---where $\eta$ is the set of germs with corank $k$ differential. The present paper is devoted to the problem of calculating Thom polynomials in the $n \leq p$ case, as well as the study of their
interior structure.

As discussed above, Thom polynomials are parameterized by an algebra
$\mathcal Q$, and two integers $n$, and $p$; we will call such a Thom polynomial
$\Tp_{\mathcal Q}(n,p)$. It turns out that the corresponding $\eta\subset
\E_0(n,p)$ is finite codimensional if and only if $\mathcal Q$ is a finite
dimensional, commutative, local algebra. We recently showed
in \cite{dstab} that---under technical conditions---the Thom polynomials $\Tp_{\mathcal Q}(n,p)$ for the same $\mathcal Q$ but varying $n$ and $p$ can be organized into a formal power series
in infinitely many variables (cf. Section \ref{sec:d-stab}). The direct application of the method of
Restriction Equations \cite{rrtp} yields certain individual Thom polynomials,
but not the calculation of whole Thom series (unless $\mathcal Q$ has very small
dimension, see Section \ref{sec:exa}).

In a recent paper \cite{bsz06} B\'erczi and Szenes introduced a new method of studying
Thom polynomials of so-called Morin singularities, ie. singularities corresponding to algebras $\mathcal Q=\C[[x]]/(x^i)$.
One of their key ideas is the usage of (improved versions of) equivariant localization formulas.
Their method naturally presents the whole Thom series. As a result, they reduced the calculation of the
Thom series of Morin singularities to a finite set of data, as well as determined this
data for $i\leq 7$. Another important novelty of \cite{bsz06} is the encoding of the Thom
series of Morin singularities by iterated residues of certain rational functions.

In the present paper we revisit a partial resolution construction of J.
Damon for all contact singularities. The B\'erczi-Szenes equivariant localization
formula applied to this construction leads to our main
result, a Localization Formula for $\Tp_{\mathcal Q}(n,p)$, see Theorem \ref{lf}.
The form of this formula implies different stabilization properties of
Thom polynomials, including the long-hidden $d$-stability property of
Thom series. The inputs of the Localization formula for a fixed $\mathcal Q$ is a
{\em finite} set of various Euler classes inside of a Grassmannian, or
Hilbert scheme; showing how a finite set can encode a whole Thom series.
Below we develop different techniques to find these Euler classes, and
hence we will calculate several new Thom series. These examples include
Thom series corresponding to local algebras of dimension $<6$, as well
as a two-parameter list of algebras.

The present work relies on the recent rapid developments of Thom polynomial theory; such as the method of
restriction equations of the authors, and the various extensions and
applications made by M. Kazarian. It is however particularly triggered
by the new ideas and results of B\'erczi and Szenes \cite{bsz06}.

\subsection{The plan of the paper} In Section \ref{sec:tpforcontact} we recall contact
singularities and give a firm foundation of their Thom polynomial
theory. In Section  \ref{sec:exa} we summarize known Thom polynomials. In Section
\ref{sec:resolution}, \ref{sec:localization} and \ref{sec:locforcontact}
 we review Damon's partial resolution of contact singularities
and explain how it yields an equivariant Localization Formula for Thom
polynomials. In Section  \ref{sec:stability} we explore stabilization properties that our
main formula implies. In Section  \ref{sec:calc} we develop geometric and algebraic
techniques to calculate the inputs of the Localization Formula. In Section  \ref{sec:return} we explain the connection with the equivariant geometry of the local punctual Hilbert scheme. Section  \ref{sec:phi} presents the calculation of Thom series of singularities $\Phi_{n,r}$. In Section \ref{sec:generatingfn} we study generating functions of Thom series, iterated residue formulas, and their relations to geometry (as well as iterated residue identities they depend on). The interesting phenomena of formally applying Thom polynomial formulas to dimensions where they are
not defined is discussed in Section \ref{sec:smallp}.

Throughout the paper we will work in the complex analytic category. Cohomology will be meant with rational coefficients.

\subsection{Acknowledgements}
The authors are indebted to M. Kazarian for several helpful discussions on the topic. He also informed us about his work in progress \cite{kaza:gysin},\cite{kaza:noas} in which he calculates Thom series using the Gysin-map. We are also grateful to P. Frenkel, B.
K\H om\H uves, T. Ohmoto, G. Smith, A. Szenes, G. B\'erczi, E. Szab\'o and C. T. C. Wall for valuable comments. The first author thanks T. Ohmoto for the opportunity to visit  Hokkaido University and RIMS which greatly helped this work.

\section{Thom polynomials of contact singularities}\label{sec:tpforcontact}

\subsection{Contact equivalence of finite germs} \label{sec:contact}
Consider $\E_0(n,p)$, the vector space of holomorphic map germs $(\C^n,0)\rightarrow (\C^p,0)$. Throughout the paper we assume that $n\leq p$. The vector space $\E_0(n):=\E_0(n,1)$ is an algebra without an identity. The space $\E_0(n,p)$ is a module over $\E_0(n)$. A map germ $g\in \E_0(n,p)$ induces a pullback $g^*:\E_0(p)\to \E_0(n)$ by composition.

\begin{definition} \label{quotientalgebra} The ideal $I_g$ of a germ $g \in \E_0(n,p)$ is the ideal in $\E_0(n)$ generated by $g^*\E_0(p)$.  The {\em quotient algebra} $Q_g$ of a germ $g \in \E_0(n,p)$ is defined by $Q_g = \E_0(n)/ I_g$.
\end{definition}

Here, and in the whole paper, an ideal {\em generated} by some ring elements is the smallest ideal containing the specified ring elements, even if the ring has no identity.  In singularity theory one usually considers the one  dimension larger {\em local algebra}---defined as $\mathcal{Q}_g:=\E(n)/ I_g$, where $\E(n)$ is the ring of function germs $(\C^n,0)\to\C$---which has an identity. The two versions can easily be obtained from each other, but the quotient algebra comes up more naturally in our geometric setting. We will be concerned with germs $g$ for which the quotient algebra is finite dimensional. We call these germs {\em finite}. Finite germs only exist for $n\leq p$, this is the reason of our overall assumption of $n\leq p$.  Finiteness is also equivalent to the property that the ideal $(g^{*}\E_0(p))$ contains a power of $\E_0(n)$. For a finite germ, in local coordinates $g=(g_1(x_1, \ldots, x_n), \cdots, g_p(x_1, \ldots, x_n))$, we have
$$\mathcal{Q}_g=\C[[x_1,\ldots,x_n]]/  (g_1, \ldots, g_p),\qquad\text{and}\qquad  Q_g=\M_n /(g_1, \ldots, g_p),$$
where $\M_n$ is the maximal ideal of $\C[[x_1,\ldots,x_n]]$, that is, the ideal generated by the variables $(x_1,\ldots,x_n)$. For finite germs the quotient algebra is nilpotent, so we will also call $Q_g$ the {\em nilpotent algebra} of the germ $g$ to distinguish it from the local algebra.

\smallskip

The invertible holomorphic germs $(\C^n,0)\to (\C^n,0)$ form the {\em right group} $\mathcal{R}(n)$. This group acts on $\E_0(n)$ by composition, and hence it acts on the set of ideals of $\E_0(n)$.

\begin{definition} \label{contactequivalent} Two germs $f,g \in \E_0(n,p)$ are {\em contact equivalent} if their ideals are in the same $\mathcal{R}(n)$ orbit. An equivalence class $\eta\subset \E_0(m,n)$ is called a (contact) singularity.
\end{definition}

In singularity theory one considers the so called contact group \cite[Ch3, 1.6]{avgl}
    \[ \mathcal{K}=\mathcal{K}(n,p)=\{(h,M):h\in \mathcal{R}(n), M \text{ is a germ}\  (\C^n,0)\to (\GL(p),1)\},\]
acting on the vector space $\E_0(n,p)$ by $\big((h,M)g\big)(x)=M(x)g(h^{-1}(x))$, and defines germs to be contact equivalent if they are in the same orbit. It is a theorem of Mather \cite[Thm 2.9]{mather4} that for finite germs the two definitions are equivalent. Hence we will denote the contact equivalence class of a germ $g$ by $\mathcal{K}g$. Equivalently
\begin{theorem}[Mather] \label{isoq-germ} The finite germs $f,g \in \E_0(n,p)$ are contact equivalent if and only if their nilpotent algebras are isomorphic.
\end{theorem}

Indeed, suppose that we have and isomorphism $\phi:Q_f\to Q_g$, and let $[x_i]$ denote the image of $x_i\in\M_n$ in $Q_f$. Pick $p_i\in\M_n$ such that $[p_i]=\phi[x_i]$ in $Q_g$. It is easy to check that the $p_i$'s can be chosen in such a way that $h=(p_1,\dots,p_n)\in \E_0(n,n)$ is an element of $\mathcal{R}(n)$ and $hI_f=I_g$.

\subsection{Thom polynomials}\label{sec:tp}

Given a map $F:N^n\to P^p$ between complex manifolds, and a point $x\in N$, we can choose charts around $x$ and $F(x)$, and consider the germ of $F$ at $x$ in these charts. The contact singularity of this germ does not depend on the choice of the charts. Indeed, it is a consequence of Mather's theorem cited above that after reparametrizing the source $(\C^n,0)$, and target $(\C^p,0)$ spaces, the ideal of the germ will be in the same $\mathcal{R}(n)$ orbit. Equivalently, we can refer to the fact that the group $\K(n,p)$ contains the group $\mathcal{R}(n)\times \mathcal{R}(p)$ of holomorphic reparametrizations of the source $(\C^n,0)$ and target spaces $(\C^p,0)$. (In this context the group $\mathcal{R}(p)$ is usually
called the left group and denoted by $\mathcal{L}(p)$.)

 Therefore it makes sense to talk about the contact singularity of $F$ at a point $x\in N$. Hence, for a map $F:N^n\to P^p$ and a contact singularity $\eta\subset \E(n,p)$, we define the singularity subsets
   $$\eta(F)=\{x\in N| \text{ the germ of $F$ at $x$ is in } \eta\}.$$

\noindent After some preparations (Sections \ref{sec:pd}-\ref{sec:jetapprox}), in Theorem~\ref{tp} we explain the following statements.
\begin{quotation}
{\em If $F$ satisfies certain transversality conditions, then the subset $\eta(F)$ defines a cohomology class $[\eta(F)]\in H^*(N)$. Moreover, this class can be expressed as a universal polynomial (the Thom polynomial, $\Tp(\eta)$) of the Chern classes of $TN$ and $f^*TP$.}
\end{quotation}
First, in Section \ref{sec:pd}, we will discuss degeneracy loci, and how universal cohomology classes are associated with them. Then we will interpret $\eta(F)$ as a degeneracy locus in Section \ref{sec:jetapprox}. These two sections serve as a rigorous definition of the Thom polynomial for a contact singularity class.

\subsection{Poincar\'e dual, equivariant cohomology and degeneracy loci}\label{sec:pd}
In this section we discuss degeneracy loci, and cohomology class represented by them. First recall that subvarieties $Y \subset X$ represent cohomology classes in the underlying space (see e.g. \cite[p.219]{fulton:young}).

 \begin{proposition}[{\bf Definition}] \label{dual} If $X$ is a smooth algebraic variety and $Y$ is an irreducible subvariety of complex codimension $d$ then there is a unique element $[Y\subset X]\in H^{2d}(X)$ such that
  \begin{enumerate}
   \item $[Y\subset X]$ is supported on $Y$, i.e. $[Y\subset X]$ restricted to $X\setminus Y$ is zero,
   \item $[Y\subset X]|_{X\setminus\sing Y}=[Y^o\subset (X\setminus\sing Y)]$.
  \end{enumerate}
 Here $\sing Y$ denotes the singular subvariety of $Y$ and $Y^o=Y\setminus \sing Y$. The cohomology class $[Y^o\subset
(X\setminus\sing Y)]$ is defined by extending the Thom class of a tubular neighborhood of the proper submanifold $Y^o\subset (X\setminus\sing Y)$ via excision.
 \end{proposition}

If $Y$ has several components $Y_i$ (usually of the same codimension) then $[Y\subset X]$ is defined to be the sum of the
classes $[Y_i\subset X]$.  When the underlying space $X$ is clear from the context, we denote $[Y\subset X]$ by $[Y]$.

\medskip

We need the equivariant version of Proposition \ref{dual} above. Let $G$ be a complex algebraic Lie group. If $X$ is a smooth algebraic variety with a $G$-action, and $Y$ is a $G$-invariant subvariety then $Y$ represents a $G$-equivariant cohomology class in the equivariant cohomology of $X$, as follows (see e.g. \cite{kazass} or \cite{edidin-graham},or a recent account, see \cite{fulton:eq}).

\begin{theorem}[{\bf Definition}] \label{equi_dual}
Let $X$ be a smooth algebraic variety with a $G$-action, and $Y\subset X$ be a $G$-invariant irreducible subvariety of complex codimension $d$. Then there is a unique element $[Y\subset X]_G\in H^{2d}_G(X)$ (called the {\em $G$-equivariant Poincar\'e dual of $Y$ in $X$}) such that for all algebraic principal $G$-bundles $\pi:P\to M$ over a smooth  algebraic variety $M$ with classifying map $k:M\to BG$ we have
   \begin{equation}[P\times_GY\subset P\times_GX]=\tilde{k}^*[Y\subset X]_G,\label{eq:gpd} \end{equation}
where $\tilde{k}:P\times_G X\to EG\times_G X$ is induced by $k$.
 \end{theorem}

Intuitively $[Y\subset X]_G$ is the class represented by $EG\times_GY$ in $EG\times_GX$. From the next section on, we will be mainly interested in the case when $X$ is a vector space. Then $H^*_G(X)\iso H^*(BG)$ canonically. The class $[Y\subset X]_G$ has various definitions and names in the literature (see \cite{zsolt} for an account). We will also use the notation $[Y]_G$ or simply $[Y]$ for $[Y\subset X]_G$, when the underlying space $X$ and the group action is clear from the context.

\begin{remark} We can make this construction more explicit for the torus $T=\GL(1)^r$ (this is the case we need in the Localization Formula below): repeat the  construction above for the principal $T$-bundle $P=\big(\C^{d+1}\setminus \{0\}\big)^r\to (\P^d)^r$. Here the classifying map $k$ is the standard inclusion $(\P^d)^r\to (\P^\infty)^r=BT$. It is not difficult to show that
   \[\tilde{k}^*:H^j(ET\times_TX)\to H^j(P\times_TX)\]
 is bijective for $j\leq 2d$. So for large enough $d$, equation (\ref{eq:gpd}) defines $[Y\subset X]_T$ uniquely:
        \[[Y\subset X]_T=(\tilde{k}^*)^{-1}[P\times_TY\subset P\times_TX].\]
\end{remark}
\begin{definition} \label{def:deg-locus} Suppose now that $s:M\to E$ is a section of the fiber bundle $\varphi:E=P\times_GX\to M$. If $Y\subset X$ is a $G$-invariant subset, then we use the notation $Y(\varphi):=P\times_GY$ for the set of `$Y$-points' in $E$ and $Y(s):=s^{-1}\big(Y(\varphi)\big)$ for the set of `$Y$-points of $s$'. We call $Y(s)$ the {\em degeneracy locus} corresponding to $Y$ and $s$.
\end{definition}

 To make a statement about the class represented by $Y(s)\subset M$ (Corollary \ref{tp_deg_locus}) we need to discuss transversality.

 \begin{definition}Let $f:A\to B$ be an algebraic map between algebraic manifolds and $Y\subset B$ be a subvariety. The map $f$ is {\em transversal} to $Y$ if it is transversal to all singularity strata of $Y$, i.e. to the (not necessarily equidimensional) manifolds $Y^o=Y\setminus \sing Y$, $\sing Y\setminus \sing(\sing Y)$ and so on.
  \end{definition}

 The following is a well known fact.
\begin{proposition}\label{trans}If $f:A\to B$ is transversal to $Y\subset B$, then $f^*([Y])=[f^{-1}(Y)]$.\end{proposition}

This statement easily generalizes to the equivariant setting.
\begin{proposition}\label{prop:equitrans}
Let the $G$-equivariant map $f:A\to B$ be transversal to the $G$-invariant subvariety $Y\subset B$. Then
$f^*([Y]_G)=[f^{-1}(Y)]_G$.
\end{proposition}

As a consequence, the equivariant class $[Y]_G$  determines the class of the degeneracy locus  $Y(s)$:

\begin{corollary}\label{tp_deg_locus} If the section $s:M\to E$ of the vector bundle $\varphi:E=P\times_GX\to M$ is transversal to $Y(\varphi)=P\times_GY$ then $[Y(s)]=k^*[Y]_G $ where $k:M\to BG$ is the classifying map of $P$.
\end{corollary}

\begin{remark} In the complex algebraic setting the existence of a transversal section is not guaranteed. Nevertheless $k^*[Y]_G $ is always an obstruction: if $k^*[Y]_G $ is non-zero then there is no section $s$ with  $Y(s)=\emptyset$, since $k^*[Y]_G $ is supported on $Y(s)$. The theory can be extended to the real smooth category. In that case the existence of the Poincar\'e dual is not automatic, but a generic section is transversal.
 \end{remark}

\subsection{Jet approximation: reduction to finite dimension} \label{sec:jetapprox} In this section we interpret $\eta(F)$ (from Section \ref{sec:tp}) as a degeneracy locus. For this we need a finite dimensional approximation of $\E_0(n,p)$, and related notions.

The vector space of $k$-jets is defined to be the vector space of degree $k$ polynomials $(\C^n,0)\to (\C^p,0)$. That is, we have
\[ J^k(n,p)=\bigoplus_{i=1}^k\Hom(\Sym^i\C^n,\C^p), \] where $\Sym^i\C^n$ is the $i$\textsuperscript{th} symmetric power of the vector space $\C^n$. Let $J^k(n)=J^k(n,1)$. The map $\E_0(n,p)\to J^k(n,p)$, defined by taking the degree $k$ Taylor
polynomial at 0, will be denoted by $j^k$. The space $J^k(n)$ is an algebra (without identity) with multiplication $h_1\cdot h_2=j^k(h_1\cdot h_2)$. The $j^k$-image (`$k$-jets') of elements in $\mathcal{R}(n)$ form a group $\mathcal{R}^k(n)$. The group $\mathcal{R}^k(p)$ will also be denoted by $\mathcal{L}^k(p)$. The group $\mathcal{R}^k(n)$ acts on the algebra $J^k(n)$ by \[ \alpha \cdot h = j^k( h \circ \alpha^{-1}) \qquad\qquad \big(\alpha\in \mathcal{R}^k(n), h\in J^k(n)\big). \]
Hence the group $\mathcal{R}^k(n)$ also acts on the set of ideals of $J^k(n)$. Similarly we can define the group $\mathcal{K}^k=\mathcal{K}^k(n,p)$ acting on the vector space $J^k(n,p)$.

Let $h\in J^k(n,p)$. The ideal in $J^k(n)$, generated by the coordinate functions of $h$, will be denoted by $I_h$. We call $Q_h=J^k(n)/I_h$ the nilpotent algebra of the jet $h$. Two $k$-jets are defined to be contact equivalent, if their ideals are in the same $\mathcal{R}^k(n)$-orbit.

\begin{proposition}\label{qiso} \cite[Thm 2.9, 2.1]{mather4} Two $k$-jets are contact equivalent if and only if they are in the same $\mathcal{K}^k$-orbit if and only if their nilpotent algebras are isomorphic.
\end{proposition}
The proof is the same as of Proposition \ref{isoq-germ}.

Next we define invariants of germs and jets. The dimension of the quotient algebra  $\mu(f):=\dim(Q_f)$ of a finite germ (or $k$-jet) $f$ plays a crucial role in our study.

We say that $f$ has depth $d$ ($\depth(f)=d$) if $d$ is the smallest $i$ for which $\E_0(n)^{i+1}\subset I_f$ (or $(J^k(n))^{i+1}\subset I_f$ for the $k$-jet case). It is an application of the Nakayama lemma that $\depth(f)\leq \mu(f)$.

\begin{definition} A germ  $f\in \E_0(n,p)$ is $k$-determined if for every $g\in \E_0(n,p)$ the germ $g$ is contact equivalent to $f$ if their $k$-jets are equal.
\end{definition}

Our main objects---the finite germs---are finitely determined due to the following
\begin{theorem}\cite{gaffney:phd},\cite[Thm 1.2]{wall:finitedet} Any finite germ (or $k$-jet) $f$ is $\depth(f)+1$-determined.
\end{theorem}

 Using the observation that
 \begin{equation}\label{eq:jet-of-k} j^k(H(g))=j^k(H)(j^k(g)) \end{equation}
for any $H\in \mathcal{K}(n,p), g\in \E_0(n,p)$) we immediately get the following
\begin{proposition}\label{reduction} Suppose that $f\in\E_0(n,p)$ is $k$-determined. Then $f$ is contact equivalent to $g$ if and only if $j^kf$ is contact equivalent to $j^kg$.
\end{proposition}
In fact, the previous statements also imply that if $k\geq\depth(f)+1$ then $Q_f\iso Q_{j^kf}$.

\smallskip

Now we can give the  degeneracy locus description of the singularity set $\eta(F)$ promised in Section \ref{sec:tp}. Given a map $F:N^n \to P^p$ between manifolds, and a positive integer $k$, we construct a fiber bundle
\begin{equation}\label{fibration}\phi_F:\{(x,h): x\in N,\ h\
               \text{is the $k$-jet of a germ $(N,x)\to(P,F(x))$}\}\ \to\ N
\end{equation}
     $$(x,h)\mapsto x,$$
with fiber $J^k(n,p)$. For $k=1$ this is the vector bundle Hom$(TN,F^*TP)$. In general the fiber is a vector space, but the structure group (the left-right group $\mathcal{R}^k(n)\times\mathcal{L}^k(p)$ acting on $J^k(n,p)$ by composition on the two sides) does not act linearly. The bundle $\phi_F$ has a natural section
\begin{equation}\label{jetsection}j^kF: x \mapsto (x,j^k(\text{germ of $F$ at $x$})).\end{equation}

Now let $g\in \E_0(n,p)$ be a $k$-determined germ, let $\eta=\mathcal{K}g$ be its contact equivalence class, and let $\eta^k$ be the contact equivalence class of $j^kg\in J^k(n,p)$. Since $\eta^k$ is  $\mathcal{R}^k(n)\times\mathcal{L}^k(p)$-invariant, it defines a degeneracy locus in the sense of Definition \ref{def:deg-locus}. Corollary \ref{reduction} implies that we have the following degeneracy locus description of the $\eta$ singularity subset of a map $F$
\begin{equation}\label{deg_locus}  \eta(F) =  \eta^k(j^k(F)).\end{equation}

\subsection{Definition of the Thom polynomial}

Now we are ready to define the Thom polynomial of a contact singularity. Let $g\in \E_0(n,p)$ be a $k$-determined finite germ. Let $\eta^k\subset J^k(n,p)$ be the closure of $\mathcal{K}^kj^kg$. Notice that $\mathcal{K}^k$ is a connected algebraic group acting algebraically, so the orbit-closure is the same in the Zariski and the metric topology. Connectedness implies that the closure of the orbit is irreducible.

\begin{definition} The Thom polynomial $\Tp(g)$ of the $k$-determined finite germ $g\in \E_0(n,p)$ (or a $k$-jet) is defined to be the class represented by $\eta^k$ in the $\mathcal{R}^k(n)\times \mathcal{L}^k(p)$-equivariant cohomology of $J^k(n,p)$.
\end{definition}
Since the contact class of $g$ depends only on the quotient algebra of $g$, we will also use the notation $\Tp_Q(n,p):=\Tp(g)$ for any $g\in J^k(n,p)$ with $Q_g\iso Q$.

$J^k(n,p)$ is a vector space (hence contractible), and $\mathcal{R}^k(n)\times \mathcal{L}^k(p)$ is homotopy equivalent to $\GL(n)\times \GL(p)$, therefore we have
\[ \Tp(g)=[\eta^k \subset J^k(n,p)]_{\mathcal{R}^k(n)\times \mathcal{L}^k(p)} =
       [\eta^k \subset J^k(n,p)]_{\GL(n)\times \GL(p)} \in H^*\big(B(\GL(n)\times \GL(p))\big). \]

The degree of the Thom polynomial is the codimension of $\eta^k$ in $J^k(n,p)$.  We will also refer to this degree as the {\em codimension of the germ $g$}.

The cohomology ring $H^*\big(B(\GL(n)\times \GL(p))\big)$ is a polynomial ring generated by the universal Chern classes $a_1,\ldots,a_n, b_1,\ldots,b_p$ of the groups $\GL(n)$, $\GL(p)$, hence the Thom polynomial is indeed a polynomial.

The meaning of the Thom polynomial is enlightened by putting together expression (\ref{deg_locus}) with Definition~\ref{equi_dual}. We obtain the following

\begin{proposition}\label{tp} Let $g\in \E_0(n,p)$ be a $k$-determined germ, $\eta^k$ the closure of $\mathcal{K}^k(j^kg)$ in $J^k(n,p)$, and let $F:N^n\to P^p$ be a map between compact complex manifolds. Suppose that the section $j^kF$ (see~(\ref{jetsection})) is transversal to $\eta^k(\phi_F)$---the  $\eta^k$-points of the fibration~(\ref{fibration}). Then the cohomology class $[\overline{\eta(F)}\subset N]$ represented by the $\eta$-points---where $\eta=\mathcal Kg$---of the map $F$ is equal to the Thom polynomial of $g$ evaluated at the Chern classes of $TN$ and $F^*TP$.
\end{proposition}

We have not included the letter $k$ in the notation $\Tp(g)$, since as we will show in Section \ref{sec:k} that the Thom polynomial does not depend on the choice of $k$, as long as $g$ is $k$-determined. Observe that Proposition \ref{tp} proves this statement, provided there are sufficiently many maps $F$ satisfying the conditions of the Proposition. We will use a different approach in Section \ref{sec:k}.

\section{Examples, known results}\label{sec:exa}

 Suppose that a complex, commutative, finite dimensional, local algebra $\mathcal{Q}$ is given.  Then there exists a contact singularity in $\E_0(n,p)$ with local algebra $\mathcal{Q}$ for each $n$, and $p$ if $n$ and $p-n$ are large enough. A general Thom polynomial is a formula (containing $n$, and $p$ as parameters) expressing the Thom polynomial of all these singularities together.
For example, the Thom polynomial of a singularity in $\E_0(n,p)$ with local algebra  $\mathcal{Q}=\C[[x]]/(x^3)$ is
\[c_{l+1}^2+\sum_{i=1}^\infty 2^{i-1}c_{l+1-i}c_{l+1+i},\]
where $l=p-n$, and the classes $c_i$ are defined by
\begin{equation}\label{quotient_vars}1+c_1t+c_2t^2+\ldots=\frac{1+b_1t+b_2t^2+\ldots+b_pt^p}{1+a_1t+a_2t^2+\ldots+a_nt^n},\end{equation}
and the conventions $c_0=1$, $c_{<0}=0$. Using Schur polynomials
\begin{equation}\label{schurdef}
\Delta_{\lambda_1\geq \lambda_2\geq \ldots\geq \lambda_r}:=\det \left( c_{\lambda_i+j-i} \right)_{i,j=1,\ldots,r}
\end{equation}
we can further write the Thom polynomial in the form
$$\Delta_{l+1,l+1}+2\Delta_{l+2,l}+4\Delta_{l+3,l-1}+\ldots.$$
This example displays four important properties:
 \begin{itemize}
  \item the Thom polynomial can be expressed in the ``quotient variables'' (\ref{quotient_vars});
  \item when the Thom polynomial is expressed in quotient variables, then the dependence on $p$ and $n$ is only through $l=p-n$;
  \item if the general Thom polynomial is expressed in quotient variables, and the indexes are shifted by $l+1$ (i.e.
           substituting $d_i=c_{l+1+i}$), the expression does not depend on $l$ either;
  \item the coefficients of Thom polynomials in the basis of Schur polynomials are non-negative.
 \end{itemize}
The general Thom polynomial after shifting the indices by substituting $d_i=c_{l+1+i}$ is called the Thom series of the local algebra $\mathcal{Q}$, and is denoted by $\ts_\mathcal{Q}$. For example
  \[  \ts_{\C[[x]]/(x^3)}=d_0^2+d_{-1}d_1+2d_{-2}d_2+4d_{-3}d_3+\ldots.  \]
Alternatively we can work with nilpotent algebras. We will also use the notation $\ts_Q$ for $Q$ being a nilpotent algebra.

All four properties above hold in general. The first three we will prove in Section \ref{sec:stability}. The first two are classical facts, we will call them the {\em Thom-Damon-Ronga theorem}, the third (in a special case) is a theorem from \cite{dstab}. The fourth property was recently proved in \cite{pragacz:positivity}.

\smallskip

Several individual Thom polynomials are known for small values of $l$ (see e.g. \cite{rrtp}, \cite{kazamulti}), but hardly any general Thom polynomials, i.e. Thom series are known. Here is a complete list of local algebras whose Thom series is known:
   \begin{itemize}
     \item $\Sigma^n=\C[[x_1,\ldots,x_n]]/\M_n^2$ (Giambelli-Thom-Porteous formula)
     \item $A_2$ \cite{rongaij}
     \item $A_3$ \cite{a3} (announced), \cite{pragacz} (sketched), \cite{bsz06}, \cite{lascoux-pragacz} (proved)
     \item $A_4,A_5,A_6$ \cite{bsz06}
     \item The Thom-Boardman classes $\Sigma^{n,1}$ \cite{kf}
     \item $I_{2,2}$ \cite{dstab, pragacz:i22, kaza:gysin}.
   \end{itemize}
Here and in what follows we use standard notations of singularity theory, as follows. $A_i=\C[[x]]/(x^{i+1})$,
$I_{a,b}=\C[[x,y]]/(xy,x^a+y^b)$, $III_{a,b}=\C[[x,y]]/(xy,x^a,y^b)$. For Thom series of Thom-Boardman classes see also Section \ref{sec:smallp}.

Below we will develop a method to calculate general Thom polynomials. It leads to formulas where the stability properties are not apparent, as these formulas are given in Chern roots. In Section \ref{sec:quotientchern} we show how to find formulas in quotient variables. In Section \ref{sec:generatingfn} we explore even more compact descriptions in terms of generating functions.

\section{A partial resolution}\label{sec:resolution}

In this section we introduce the key geometric idea leading to our cohomological localization formula. We present a partial resolution (i.e. a birational map of varieties) of contact invariant subvarieties of $J^k(n,p)$, in particular of closures of contact singularities. The construction is originally due to J.~Damon~\cite{damonphd}. Similar ideas are present in works of J.~Mather on Thom-Boardman singularities~\cite{mathertb}. The idea is that a contact invariant subvariety of $J^k(n,p)$ is the union of large linear subspaces.

Let $m$ be a nonnegative integer, and let $\gr^m=\gr^m(J^k(n))$ be the Grassmannian of $m$-codimensional linear subspaces of $J^k(n)$.

\begin{definition} \label{def:correspondence} Let $Y\subset \gr^m$ be a subvariety and fix $p\geq 1$. The {\em correspondence variety} of $Y$ is
  \[  C(Y)=\{(I,g)\in\gr^m\times J^k(n,p)\ |\ I\in Y, I_g\subset I\}.  \]
\end{definition}
We have now the following diagram
\begin{equation} \label{damondiag}
\xymatrix{C(Y)  \ar@{^{(}->}@<-3pt>[rr]^{i}  \ar[d] &  & \gr^m\times
J^k(n,p) \ar[d]^{\pi_1}\ar[r]^{\ \ \ \pi_2} & J^k(n,p) \\
Y\ar@{^{(}->}@<-3pt>[rr]& & \gr^m, & } \end{equation} where $\pi_1$,
$\pi_2$, $i$ are the obvious projections and imbedding.
 The projection $C(Y)\to Y$ makes $C(Y)$ a vector bundle with fiber $C_I=I\otimes\C^p\subset J^k(n)\otimes\C^p=J^k(n,p)$.

\begin{proposition}\label{birat} Let $g \in J^k(n,p)$ with $\mu(g):=\codim I_g=m$. Let $Y$ be $\overline{\mathcal{R}I_g}\subset \gr^m$ for  $\mathcal{R}=\mathcal{R}^k(n)$.  Then
\[\phi=\pi_2\circ i:C(\overline{\mathcal{R}I_g})\to \overline{\mathcal{K}g}\]
is a birational map.
\end{proposition}
\begin{proof} For $\tilde{\mathcal{K}}g:=\{(I_h,h):h\in \mathcal{K}g\}$ we see that $\phi|_{\tilde{\mathcal{K}}g}:\tilde{\mathcal{K}}g\to \mathcal{K}g$ is a bijection, so it is enough to show that $\tilde{\mathcal{K}}g\subset C(\overline{\mathcal{R}I_g})$ is open (and therefore dense since $C(\overline{\mathcal{R}I_g})$ is irreducible). For this it is enough to show that $\tilde{\mathcal{K}}g$ intersected with a fiber is open in the fiber; hence we need the following lemma.

\begin{lemma}\label{ei}For any jet $g \in J^k(n,p)$ the set $A_g:=\{h\in I_g\otimes\C^p:I_h=I_g\}$ is Zariski open in $I_g\otimes\C^p$.\end{lemma}

 \begin{proof} Let $h=(h_1,\ldots,h_p)\in I_g\otimes\C^p$, and let $a_i^j$ be the coefficients of $h_i$ in some linear
basis of $I_g$. The property that the $h_i$'s generate $I_g$ as an ideal is equivalent to the property that an appropriate matrix, whose entries are linear functions of the $a_i^j$'s, has full rank. Therefore, the property that the $h_i$'s generate $I_g$ cuts out a Zariski open subset.
\end{proof}
This finishes the proof of Proposition \ref{birat}.
\end{proof}
Using the Gysin (or pushforward) map $\phi_*$ we have that $\phi_*(1)=[\tilde{\mathcal{K}}g]=\Tp(g)$. Details on the properties of the equivariant Gysin map can be found in \cite{fulton:eq}. We calculate the Gysin map using localization in the next section.

\section{Singular-base equivariant localization}\label{sec:localization}

In this section we recall a version of the Berline-Vergne-Atiyah-Bott equivariant localization formula, due to B\'erczi and Szenes. For completeness we give a proof. This version presents the `localization' of an equivariant cohomology class on the total space of a vector bundle over a compact singular base space.

Let $V$ be a vector space. Suppose that $M$ is a compact algebraic manifold, and $Y\subset M$ a subvariety. Let $E \to Y$ be a sub-vector bundle of the trivial bundle $M\times V \to M$ restricted to $Y$. Let $\pi_2: M\times V \to V$ be the projection, $i:E \subset M\times V$ the embedding, and $\phi=\pi_2\circ i$, as in the diagram
  $$\xymatrix{E  \ar@{^{(}->}@<-3pt>[rr]^{i} \ar@/^2pc/[rrr]^{\phi} \ar[d] &  & M\times
                V \ar[d]\ar[r]^{\pi_2} & V \\ Y\ar@{^{(}->}@<-3pt>[rr]&  & M. & } $$
Recall that for $A\subset B$, by $[A]$ or $[A\subset B]$ we mean the cohomology class represented by $A$ in the cohomology of $B$.

\begin{proposition} \cite[(3.8)]{bsz06} \label{31} Suppose that the torus $T$ acts on all spaces in the diagram above, and that all maps are $T$-equivariant. Assume that the fixed point set $F(M)$ of the $T$-action on $M$ is finite.
Then for the push-forward map $\phi_*:H_T^*(E)\to H_T^*(V)$ we have
\begin{equation} \label{singbase}
\phi_*(1) =\sum_{f\in F(M)}\frac{[Y\subset M]|_f\cdot [E_f\subset V]}{e(T_fM)}=
           \sum_{f\in F(Y)}\frac{[Y\subset M]|_f\cdot [E_f\subset V]}{e(T_fM)}.
\end{equation}
Consequently, if $\phi$ is birational to its image, then the right hand side of (\ref{singbase}) is equal to $[\phi(E)]\in H_T^*(V)$.
\end{proposition}

\begin{proof}We have to calculate the integral ${\pi_2}_*(i_*(1))=\int_{M}i_*(1)=\int_M[ E\subset M\times V]$ (we identify the cohomology of $M\times V$ with the cohomology of $M$) for which we  apply the Berline-Vergne-Atiyah-Bott localization formula, that we recall now.

\begin{proposition}\cite{atiyah-bott}\label{ab} Suppose that $M$ is a compact manifold and $T$ is a torus acting smoothly on $M$, and the fixed point set $F(M)$ of the $T$-action on $M$ is finite. Then for any cohomology class $\alpha\in H_T^*(M)$
\begin{equation}\label{abformula}\int_M\alpha=\sum_{f\in F(M)} \frac{\alpha|_f}{e(T_{f}M)}.\end{equation}
Here $e(T_fM)$ is the $T$-equivariant Euler class of the tangent space $T_fM$. The right side is considered in the fraction field of the polynomial ring of $H^*_T($point$)=H^*(BT)$ (see more on details in \cite{atiyah-bott}): part of the statement is that the denominators cancel when the sum is simplified.
\end{proposition}

\noindent We complete the proof of Proposition~\ref{31} by noticing that

\begin{equation} [E\subset M\times V]|_f=[E_f\subset V]\cdot[Y\subset M]|_f,\label{eq:product} \end{equation}
where $f\in M\subset M\times V$.
The second equality in (\ref{singbase}) follows from the fact that the cohomology class $[Y\subset M]$ is supported on $Y$, so other fixed points give zero contribution.
\end{proof}

Notice that the same argument gives a localization formula for the case of smooth base and singular fiber.

\begin{remark}\label{zero} If $\phi$ decreases the dimension, then $\phi_*(1)$ is zero---since supported on a too small subset---so the right hand side of (\ref{singbase}) is zero.
\end{remark}


\begin{remark} \label{smooth} If $f$ is a smooth point of $Y$, then
\[ \frac{e(T_fM)}{[Y\subset M]|_f} =e(T_f Y),\]
hence if $Y$ is smooth then formula (\ref{singbase}) further simplifies to
\[ \phi_*(1) =\sum_{f\in F(M)}\frac{[E_f\subset V]}{e(T_fY)}.\]
This formula holds in the general case too, if we define the {\em virtual (tangential) Euler class} $e(T_f Y)$ to be $\ \displaystyle{\frac{e(T_fM)}{[Y\subset M]|_f}}\ $ even if  $f$ is not a smooth point of $Y$.
\end{remark}

\begin{remark} The moral of Proposition~\ref{31} is that if we want to calculate the equivariant class of a variety with
localization, we should look for high-dimensional linear spaces in it. Precisely saying, we need an other variety, birational to the original, which is the total space of a vector bundle over a compact base space. The higher the rank of the bundle, the simpler the formula is. Usually the variety we start with is a cone, hence there is a canonical line bundle whose total space is birational to it. Therefore formula (\ref{singbase}) can be applied to find the Thom polynomial. This case is used in \cite{root}. Certain quiver varieties are birational to total spaces of vector bundles with higher rank, over a smooth compact space \cite{reineke}. Hence formula~(\ref{singbase}) can be effectively applied, yielding similar formulas for quiver polynomials as in \cite{ks}.
\end{remark}

\section{Localization for contact classes}\label{sec:locforcontact}

Now we apply the equivariant localization formula above to a the construction of Section \ref{sec:resolution}. This is different from the resolution used in \cite{bsz06} for Morin sinmgularities; it is more general (it covers all contact singularities), but numerically less effective.

Let $G(n,p)=\GL(n)\times \GL(p)$. Recall that the spaces in diagram (\ref{damondiag}) have $G(n,p)$-actions, and the maps in the diagram are
$G(n,p)$-equivariant. Let $T(n,p)=T(n)\times T(p)\cong U(1)^n\times U(1)^p$ be the maximal torus of $G(n,p)$, and restrict the action on the spaces and maps of diagram~(\ref{damondiag}) to $T(n,p)$. Recall also that the map $H^*_{G(n,p)}(\pt)\to H^*_{T(n,p)}(\pt)$ is injective (splitting lemma), hence by this restriction we do not loose any cohomological information. Now we can apply Proposition~\ref{31} and Proposition~\ref{birat} to the diagram (\ref{damondiag}), and we obtain our main result. Let $F$ denote the set of monomial ideals in $\gr^m$ being the fixed points of the $T(n,p)$-action on $\gr^m$ (where $T(p)$ acts trivially).

\begin{theorem}[Localization Formula] \label{lf} Let $g\in J^k(n,p)$ be a $k$-jet. Then
 \[\Tp(g)=\sum_{I\in F}\frac{[C_I\subset J^k(n,p)]\cdot[ \overline{\mathcal{R}I_g}\subset\gr^m]|_I}{e(T_I\gr^m)}.\]
\end{theorem}
\hfill\qed

Using the virtual tangent Euler classes we get
\begin{equation}\label{eq:lfeu} \Tp(g)=\sum_{I\in F}\frac{[C_I\subset J^k(n,p)]}{e(T_I\overline{\mathcal{R}I_g})}.
\end{equation}

In the rest of this section we present two lemmas in which we study the factors $[C_I\subset J^k(n,p)]$ and $e(T_I\gr^m)$ in the Localization Formula. For this we choose the following notations. Let
 \[H_{T(n,p)}^*(J^k(n,p))= H_{T(n,p)}^*(\pt)=\Z[\alpha_1,\ldots,\alpha_n,\beta_1,\ldots,\beta_p],\]
where $\alpha_i$ (resp $\beta_i$) denotes the universal first Chern class in the  $i$'th factor of $H^*_{T(n)}($pt$)=\otimes_{i=1}^n H^*(BU(1))$ (resp. $H^*_{T(p)}($pt$)=\otimes_{i=1}^p H^*(BU(1))$). We call the $\alpha_i$'s and the $\beta_i$'s the Chern roots of the group $T(n,p)$. As usual, we identify weights of a $T(n,p)$-representation with linear combinations of the Chern roots. For a $T(n,p)$-representation $A$, let $W_A$ denote the multiset of its weights. The Euler class of a representation $A$ is $e(A)=\prod_{w\in W_A} w$.

We define resultants by $\res(S|T)=\prod_{s\in S, t\in T} (s-t)$ for the finite multisets $S$ and $T$. For example, the representation of $T(n,p)$ on the vector space $\Hom(\C^n,\C^p)$ by $(A,B)\cdot F=B\circ F \circ A^{-1}$ has weights $W_{\Hom(\C^n,\C^p)}=\{\beta_i-\alpha_j:i=1,\dotsc,p; j=1,\dotsc,n\}$ and Euler class $\res(\{\beta_1,\ldots,\beta_p\}|\{\alpha_1,\ldots,\alpha_n\})$. (In the rest of the paper, we will drop the brackets $\{\ \}$ from the notation.) Similarly,
  \[  W_{J^k(n,p)}=\{\beta_i-\sum_{j=1}^n a_{ij}\alpha_j:i=1,\dotsc,p; a_{ij}\geq0, 1\leq \sum_{j=1}^n a_{ij}\leq k\} \]
 and for the $T(n)$-representation on $J^k(n)$ we have
  \[  W_{J^k(n)}=\{-\sum_{j=1}^n a_{j}\alpha_j: a_{j}\geq0, 1\leq \sum_{j=1}^n a_{j}\leq k\}.  \]

The equivariant cohomology class represented by an invariant linear subspace in a representation space is the Euler class of the factor representation. Hence we have the following lemma.

\begin{lemma} Let $I$ be a monomial ideal.  Then
\[ [C_I\subset J^k(n,p)]=e(Q_I\otimes \C^p)=\prod_{i=1}^p\prod_{w\in  W_{Q_I}}(\beta_i+w)=\res(\beta_1,\dotsc,\beta_p|-W_{Q_I}),\]
where $Q_I$ is the quotient space $J^k(n)/I$. If $I$ is monomial then $Q_I$  is equipped with the induced representation of $T(n)\leq \GL(n)$.
\end{lemma} \qed

Notice that all coefficients in $W_{Q_I}$ are negative, so in applications the form $\res(\beta_1,\dotsc,\beta_p|-W_{Q_I})$ seems more natural.

The tangent bundle of a Grassmannian is Hom$(A,B)$ where $A$ and $B$ are the tautological sub- and quotient bundles. Therefore the following lemma calculates the denominator of the Localization Formula explicitly.

\begin{lemma} We have \[e(T_I\gr^m)=\res(W_{Q_I}|W_I).\] \end{lemma} \qed

Again, if we want positive coefficients, we can write $e(T_I\gr^m)=\res(-W_{I}|-W_{Q_I})$.

The factor $[\overline{\mathcal{R}I_g}\subset \gr^m]|_I$ in the Localization Formula  is a subtle invariant of the set
 $\overline{\mathcal{R}I_g}$ at~$I$. Its calculation is difficult in general. In Section \ref{sec:calc} we calculate special cases.

\subsection{First application of the Localization Formula}\label{firstex} Let $g\in J^k(n,p)$ be the jet with the degree $d$ monomials as coordinate functions for $k\geq d$ and $p=\binom{n+d-1}{d}$. That is,
$I_g=(J)^d$, where $J=J^k(n)$. The ideal $I_g$ is a fixed point of the right group $\mathcal{R}$, so the localization formula immediately gives
\begin{equation} \label{eq:firstex}   [(J)^d(n,p)]=\res(\beta_1,\dotsc,\beta_p|-W_{Q((J)^d)}),  \end{equation}
where  $-W_{Q((J)^d)}=\{\sum a_i\alpha_i:a_i\geq0, 0<\sum a_i<d\}$.

The singularity whose $k$-jet is $g$ is called the Thom-Boardman singularity $\Sigma^{n,\dotsc,n}(n,p)$ (the number of $n$'s in the superscript is $d$). Hence (\ref{eq:firstex}) is the Thom polynomial of this singularity. This is not a new result though it might be the first appearance in the literature in this generality. The $d=1$ case recovers a special case of the Giambelli-Thom-Porteous formula (cf. Theorem \ref{porteous}) in the Chern root format:
$\Tp_{\Sigma^{n}}(n,p)=\res(\beta_1,\ldots,\beta_p|\alpha_1,\ldots,\alpha_n)$.

       \section{Stability properties of the Thom polynomial}\label{sec:stability}

   \subsection{Dependence of Thom polynomials on $k$} \label{sec:k}
Our definition of the Thom polynomial of a finite germ $f$, used its $k$-jet. In this section we show that the Thom polynomial does not depend on $k$, that is, we prove the following
\begin{theorem} \label{k} Suppose that $g\in J^l(n,p)$ and $\depth(I_g)=k\leq l$, i.e. $(J)^{k+1}\subset I_g$, where $J=J^l(n)$. Then
\[\Tp(j^k g)= \Tp(g),\]
where $j^k:J^l(n,p)\to J^k(n,p)$ is the projection.
\end{theorem}
\proof  We need to show that
\[(j^k)^{-1}\overline{\mathcal{K}^k(j^k g)}=\overline{\mathcal{K}^l(g)}.\]
 Using the same notation $j^k$ for the projection $\mathcal{K}^l\to\mathcal{K}^k$ we have that
 \begin{equation}\label{eq:jet-of-k-l} j^k(H(f))=j^k(H)(j^k(f)), \end{equation}
where $H\in \mathcal{K}^l$ and $f\in J^l(n,p)$, which implies that $j^k(\mathcal{K}^lf)=\mathcal{K}^kj^k(f)$. Therefore it is enough to show that $j^k(f)=j^k(g)$ implies $f\in \overline{\mathcal{K}^l(g)}$. Th equality $j^k(f)=j^k(g)$ implies that $ I_f+(J)^{k+1}=I_g+(J)^{k+1}$ and by the assumption on $g$ we have that $I_g+(J)^{k+1}=I_g$, which implies that $I_f\subset I_g$, which further implies $f\in \overline{\mathcal{K}^l(g)}$ by Lemma \ref{ei}.\qed

   \subsection{Dependence of Thom polynomials on $p-n$} \label{sec:p-n} In this section we prove the classical stability result---Theorem \ref{stab}---on Thom polynomials of contact singularities. Let $\sigma:J^k(n,p)\to J^k(n+1,p+1)$ denote the stabilization map
\[\sigma g(x_1,\dotsc,x_{n+1}):=\big(g_1(x_1,\dotsc,x_{n}),\dotsc,g_p(x_1,\dotsc,x_{n}),x_{n+1}\big). \]
\begin{theorem}[Stability] \label{stab}
\[\Tp(g)=\sigma^*\Tp(\sigma g),\]
where $\sigma^*:H^*_{G(n,p)}\to H^*_{G(n+1,p+1)}$ is the homomorphism induced by the map
\[ G(n,p)\to G(n+1,p+1),\ \  (M,N)\mapsto \left(\smx M001,\smx N001\right). \]
\end{theorem}
The reason we include the proof here is twofold. First, we would like to strengthen the stability theorem and show that these Thom polynomials are {\em supersymmetric}. Second, it gives us a chance to study the geometry related to the Localization Formula.

The Localization Formula gives the Thom polynomial in Chern roots i.e. in generators of $H^*_{T(n,p)}\iso \Z[\alpha_1,\dotsc,\alpha_n,\beta_1,\dotsc,\beta_p]$. Since $\Tp(g)$ is in the image of $H^*_{G(n,p)}\to H^*_{T(n,p)}$, it is symmmetric in both the $\alpha$ and the $\beta$ variables. However it has more symmetry.
\begin{definition}The polynomial $q\in \Z[\alpha_1,\dotsc,\alpha_n,\beta_1,\dotsc,\beta_p]$ is {\em supersymmetric} (see \cite{lascoux}) if
  \begin{enumerate}
  \item symmmetric in both the $\alpha$ and the $\beta$ variables,
  \item $q(\alpha_1,\dotsc,\alpha_{n-1},t,\beta_1,\dotsc,\beta_{p-1},t)$ does not depend on $t$.
  \end{enumerate}
\end{definition}
\begin{theorem} \label{supersym} The Thom polynomial of a finite germ is supersymmetric.\end{theorem}
An important property of supersymmetric polynomials is that they can be expressed in {\em quotient variables}:
We define a map
\[\rho_{n,p}:\Z[c_1,\dotsc,c_i,\dotsc]\to \Z[\alpha_1,\dotsc,\alpha_n,\beta_1,\dotsc,\beta_p]\]
by the formal power series
\begin{equation}\label{eq:c} 1+c_1t+c_2t^2+\cdots=\frac{\prod_{j=1}^p(1+t\beta_j)}{\prod_{i=1}^n(1+t\alpha_i)},
   \end{equation}
i.e. $\rho_{n,p}(c_1)=\beta_1+\cdots+\beta_p-\alpha_1-\cdots-\alpha_n$ and so on. We say that $q\in \Z[\alpha_1,\dotsc,\alpha_n,\beta_1,\dotsc,\beta_p]$ can be expressed in quotient variables if it is in the image of $\rho_{n,p}$.
\begin{theorem}[Lascoux \cite{lascoux}] \label{super=>quot}The polynomial $q\in \Z[\alpha_1,\dotsc,\alpha_n,\beta_1,\dotsc,\beta_p]$ is supersymmetric if and only if it can be expressed in quotient variables. \end{theorem}

The expression of supersymmetric polynomials in terms of the quotient variables is unique if the degree of $q$ is not too high compared to $n$ and $p$.

\begin{proposition}\label{injectivity} If $\rho_{n,p}(h)=0$ for a non-zero polynomial $h\in \Z[c_1,\dotsc,c_i,\dotsc]$ then $\deg(h)\geq (n+1)(p+1)$ with the convention $\deg c_i=i$. \end{proposition}

In fact, the kernel of $\rho_{n,p}$ is known explicitly (see \cite[\S 4.2]{pragacz:enumgeo}): $\ker (\rho_{n,p})=\langle
\Delta_\lambda:{(n+1)}^{(p+1)}\subset \lambda\rangle$, where $\langle\ \rangle$ means the generated $\Z$-module (for the definition of $\Delta$ see (\ref{schurdef})).

Now we translate supersymmetry to geometry. Notice that the stabilization map  $\sigma:J^k(n,p)\to J^k(n+1,p+1)$ is $G'$-equivariant for $G'=G(n,p)\times \GL(1)$, where $\GL(1)$ acts trivially on $J^k(n,p)$ and {\em diagonally on the last variables} of elements of $J^k(n+1,p+1)$. Supersymmetry and stability is equivalent to the following strengthening of the stability Theorem \ref{stab}:
\begin{theorem}(Strong stability.) \label{strongstab} For any $g\in J^k(n,p)$
      \[  \Tp_{G'}(g)=\Tp_{G'}(\sigma g).  \]   \end{theorem}
This theorem immediately follows from the two lemmas below and Proposition~\ref{prop:equitrans} on the transversal pull back of Thom polynomials.
\begin{lemma}\label{K-trans} The stabilization map $\sigma$ is transversal to every contact class in  $J^k(n+1,p+1)$.
\end{lemma}

\begin{lemma}\label{preim} We have $\sigma^{-1}(\mathcal{K}\sigma g)=\mathcal{K} g$ for any $g\in J^k(n,p)$.   \end{lemma}
\proof[Proof of Lemma \ref{K-trans}.] The results of Section \ref{sec:resolution} imply that for $g\in J^k(n,p)$ the tangent space of its contact class is
\begin{equation}\label{eq:K-tangent}T_g\mathcal{K}g=I_g\otimes \C^p+T_g\mathcal{R}g, \end{equation}
where
\begin{equation}T_g\mathcal{R}g=\left\{\sum_{i=1}^n p_i\partial_ig:p_i\in J^k(n)\right\}.\end{equation}
Applying this to the germ $\sigma g\in J^k(n+1,p+1)$ we see that transversality is equivalent to the property that the three subspaces $I_{\sigma g}\otimes \C^{p+1}, \ T_{\sigma g}\mathcal{R}\sigma g$ and $\sigma J^k(n,p)$ span $J^k(n+1,p+1)$.

Let $h=\sum_{i=1}^{p+1}h_i\otimes e_i$ any element of $J^k(n+1,p+1)$, where $\{e_i:i=1,\dotsc,p+1\}$ is the standard basis of $\C^{p+1}$ and $h_i\in J^k(n+1)$. We can write $h_i$ in the form $h_i=a_i+b_ix_{n+1}$ where $a_i \in J^k(n)$ and $b_i \in J^k(n+1)$. Since $x_{n+1}\in I_{\sigma g}$, it is enough to show that $a_i\otimes e_i$ is in the span. For $i\leq p$ we have $a_i\otimes e_i\in \sigma J^k(n,p)$ and for the last coordinate notice that $a_{p+1}\partial_{n+1}\sigma{g}=a_{p+1}\otimes e_{p+1}$.\qed

\proof[Proof of Lemma \ref{preim}.] The statement follows from Proposition \ref{qiso} and that $Q_{\sigma g}\iso Q_g$ for any $g\in J^k(n,p)$.\qed

The proof of Theorem \ref{strongstab}---and hence Theorems \ref{stab} and \ref{supersym}---is complete. These facts  imply the {\em Thom-Damon-Ronga theorem} of Section \ref{sec:exa}. Analogously to Thom polynomials of contact classes it is possible to define Thom polynomials for {\em right-left} classes and they cannot be expressed in quotient variables in general.
Recall that $\mu(g)=\codim (I_g \subset J^k(n))$.

\begin{proposition}\label{unique} Let $g\in J^k(n,p)$ with  $n\geq \mu(g)-1$. Then there is a unique polynomial $\tp(g)\in \Z[c_1,c_2,\dots]$ such that $\rho_{n,p}(\tp(g))=\Tp(g)$.
 \end{proposition}
\proof Theorem \ref{super=>quot} implies existence. Since $\deg(\tp(g))=\mu(g)p-\dim(\mathcal{R}g)$, we have $\deg(\tp(g))<(n+1)(p+1)$, and
Proposition \ref{injectivity} implies uniqueness. \qed

\begin{definition} If the condition $n\geq \mu(g)-1$ is not satisfied, then we can take an iterated stabilization of $g$ to get a unique polynomial, what we will also denote by $\tp(g)$. Also, we will use the notation $\tp_Q(l):=\tp(g)$, where $g\in J^k(n,p)$, such that its nilpotent algebra $Q_g$ is isomorphic to $Q$, and $p-n=l$. Stability justifies this notation.
\end{definition}

As we already remarked, formula (\ref{eq:firstex}) implies that for $p\geq \binom{n+1}{2}$
\begin{equation}\label{por_root} \Tp_{\Sigma^n}(n,p)=\res(\beta_1,\dotsc,\beta_p|\alpha_1,\dotsc,\alpha_n).\end{equation}
Since $\mu(\Sigma^n(n,p))=n$, Proposition \ref{unique} can be applied. We obtain that the polynomial
in quotient variables which is equal to the right hand side of (\ref{por_root}) will express the Thom polynomial of any $\Sigma^n(*,*+l)$
(at least for $l\geq \binom{n}{2}$). This argument reproves the following classical theorem.
\begin{theorem}[Giambelli-Thom-Porteous]\label{porteous} The Thom polynomial of $\Sigma^n$ in quotient variables is
  \begin{equation}\label{eq:gia}
     \tp_{\Sigma^n}(l)=\Delta_{{\underbrace{\scriptstyle{n+l,\dotsc,n+l}}_{\mbox{$\scriptstyle{n}$}}}}=\det(c_{n+l+j-i})_{1\leq i,j\leq n}. \end{equation}
  \end{theorem}
\subsection{Dependence of Thom polynomials on $p$}\label{sec:d-stab}
In this section we study the relation between the Thom polynomial of the jets $g$ and $\delta g$ where \[\delta:J^k(n,p)\to J^k(n,p+1),\ \ \delta g:(x_1,\ldots,x_n)\mapsto(g_1,\dotsc,g_p,0).\]
In other words we are interested in the dependence of the Thom polynomial $\tp_Q(l)$ on $l$. In \cite{dstab} we showed that under a technical condition one can calculate $\tp(g)$ from $\tp(\delta g)$ by `lowering the indices'. Notice that $Q_{\delta g}\iso Q_g$. Consequently the Thom polynomials of all germs with a given quotient algebra $Q$ (or local algebra $\mathcal{Q}$) can be arranged into a series what we called the {\em Thom series} of $Q$. The variables of this series are normalized Chern classes what we denoted by $d_i$, and hence this stabilization property will be called {\em d-stability}.

\begin{definition}\label{flat} Fix $m \in\N$ and assume that the polynomial $q\in \Z[c_0,c_1,\dotsc]$ has {\em width} $m$, i.e.
\[q=\sum_{|K|=m}a_Kc^K,\ \  \text{where} \ K\in\N^m\ \text{and} \ c^K=\prod_{i=1}^m c_{K_i}, \]
using the $c_0=1$ convention. We define the {\em lowering operator} $\flat=\flat(m)$ by
\[q^\flat:=\sum_{|K|=m}a_Kc^{K^\flat},\ \  \text{where} \ K^\flat_i=K_i-1,\]
using the $c_{-1}=0$ convention.
\end{definition}
E.g. for $m=2$ and $q=c_2^2+c_1c_3+2c_4$ we have $q^\flat=c_1^2+c_2$, where we did not write out the $c_0=1$ factors.
\begin{theorem}\label{th:d-stab} Let $g$ be a jet with $\mu(g)=m$. Then $\tp(g)$ has width $m$ and
\[\tp(\delta g)^\flat=\tp(g).\]
\end{theorem}
A simple calculation shows (see \cite[2.3]{dstab}) that to prove Theorem \ref{th:d-stab} it is enough to prove the following.
\begin{proposition}\label{root-d-stab} Let g$\in J^k(n,p)$ with $\mu(g)=m$ and write
\[ \Tp(\delta g)(\alpha_1,\dotsc,\alpha_n,\beta_1,\dotsc,\beta_p,\beta_{p+1})=\sum p_i\beta_{p+1}^{m-i}\ \ \ \text{for}
 \ \ \ p_i\in\Z[\alpha_1,\dotsc,\alpha_n,\beta_1,\dotsc,\beta_p].\]
Then $p_0=\Tp(g)$.
 \end{proposition}
\proof Notice that changing $g$ to $\delta g$ in the Localization Formula only changes the factors $[C_I]$ by multiplying them with $\res(\beta_{p+1}|-W_{Q_I})$, i.e. if
\[  \Tp(g)=\sum_{I\in F} a_I\ \ \text{for $a_I$ being the local contribution at the fixed point $I$, then}\]
     \[Tp(\delta g)=\sum_{I\in F} a_I\res(\beta_{p+1}|-W_{Q_I}), \]
which implies the Proposition, therefore the Theorem. \qed

We can rephrase Theorem \ref{th:d-stab} in terms of Thom series. Let $Q$ be an $m$-dimensional nilpotent algebra over $\C$ and let $l\geq0$ such that there exist a jet $g(n,p)\in J^k(n,p)$ with $Q_g\iso Q$ and $l=p-n$. If $k>m$, $n\geq m$ and  $p\geq b(Q)-a(Q)+n$---where $a(Q)$ is the minimal number of generators for $Q$ and $b(Q)$ is the minimal number of relations for $Q$---then such a jet exists.
\begin{theorem}\label{thomseries} Let $Q$ be an $m$-dimensional nilpotent algebra over $\C$. Then there is a unique homogeneous ($\deg d_i=i$) formal power series
\[\ts_Q=\sum_{|K|=m}a_Kd^K\in \Z[[\dotsc,d_{-i},\dotsc,d_0,\dotsc,d_i,\dotsc]],  \]
such that the Thom polynomials $\tp_Q(l)$ can be obtained by substituting $d_i=c_{i+l+1}$ with the usual $c_0=1$, $c_i=0$ for $i<0$ conventions.  \qed
\end{theorem}
This is an improvement of \cite[Th.4.1]{dstab}, where the assumption of non-zero normal Euler class was assumed. The degree of $\ts_Q$ can be calculated by finding the degree of $\Tp_Q(n,p)$ for some $n$ and $p$ which requires the calculation of the dimension of an $\mathcal{R}$-orbit (or, equivalently, of the corresponding unfolding space).

From this proof we can see that the d-stability property is not as deep as stability. It is a curious fact of the history of Thom polynomials that it remained hidden for so long.

\begin{remark} \label{m-1-eleg} The Localization Formula and Proposition \ref{unique} shows that if  we know the tangent Euler classes of $\overline{\mathcal{R}I_g}$ for $g\in J^k(m-1,p)$ with $\mu(g)=m$ at the monomial ideals (the number of these depend only on  $m$ and not on $p$, for a precise formulation, see Proposition \ref{pmax}), then we have a simple algorithm to calculate the Thom series of $Q_g$. But to give a closed formula for $\ts_Q$ in the $d$-variables is a different problem in algebraic combinatorics. The difficulty is to move from formulas in Chern roots to formulas in Chern classes. There are many unsolved problems in this context, like expressing Chern classes of various vector bundle constructions in terms of the Chern classes of the input bundles (see e.g. \cite[\S.2]{pragacz_dd}). It is sometimes a question of taste which form one prefers. In some cases we succeeded to find formulas in the $d$-variables, see Section \ref{sec:quotientchern}.
\end{remark}

\section{Further calculations}\label{sec:calc}

\subsection{Extrapolation}\label{sec:interpol}\mbox{}
The  tangent Euler classes $e(T_I\overline{\mathcal{R}I_g})$ are difficult to calculate directly. At this point we do not have a general method to do it. One of our strategies is to use the Localization Formula backwards: knowing the Thom polynomial $\Tp_Q(n,p)$ for some $n$ and $p$ we can calculate the tangent Euler classes and then we can calculate the whole Thom series. This method is based on relating the tangent Euler classes to incidences (in the sense of \cite{rrtp}).

We will use the following shorthand notations for the tangent Euler classes:
\[ e(g,f)= e(T_{I_f}\overline{\mathcal{R}I_g}),\ e(g,I)= e(T_{I}\overline{\mathcal{R}I_g}),\  e(Q,I)=e(T_I\eta_Q),\]
where $\eta_Q\subset\gr^m$ is the closure of the set of ideals $I$ with quotient algebra $Q_I=J^k(n)/I$ isomorphic to $Q$ and $\dim(Q)=m$. We also use the $Q_f=Q_{I_f}$ notation.

\begin{theorem}[Interpolation Formula]\label{interpol} Let $g\in J^k(n,p)$ and let $f\in J^k(n,p)$ be a monomial germ with $\mu(f)=\mu(g)$. Then
\[  e(g,f)=\frac{\res(W_f|W_{Q_f})}{\Tp(g)|_f},  \]
where $W_f=\{w_1,\dotsc,w_p\}$, $w_i=\sum w_{i,j}\alpha_j$ with $f_i=\prod x_j^{w_{i,j}}$ and $|_f$ denotes the restriction to the $n$-dimensional subtorus $T(f)$ of $\mathcal{K}$ fixing $f$, identifying the generators of $H^*_{T(f)}$ with $\alpha_1,\dotsc,\alpha_n$. In other words $\alpha_i|_f=\alpha_i$ and $\beta_i|_f=\sum w_{i,j}\alpha_j$.
\end{theorem}
\proof Restricting the Localization Formula we obtain
\[ \Tp(g)|_f=\sum_{I\in F}\frac{\res(W_f|-W_{Q_I})}{e(g,I)}. \]
If $I$ is a monomial ideal different from $I_f$ with $\mu(I)=\mu(f)$ then there is a $w_i\in -W_f\cap W_{Q_I} $ therefore $\res(W_f|-W_{Q_I})=0$. \qed

The next lemma will further simplify our calculations by allowing us to use as small $n$ as possible. Recall that the stabilization map $\sigma:J^k(n,p)\to J^k(n+1,p+1)$ is defined by
\[\sigma g(x_1,\dotsc,x_{n+1})=\big(g_1(x_1,\dotsc,x_{n}),\dotsc,g_p(x_1,\dotsc,x_{n}),x_{n+1}\big).\]
\begin{lemma}[Tangent Lemma]\label{tangentlemma} Let $f,g\in  J^k(n,p)$ and let $f$ be a monomial germ. Then
\[e(\sigma g,\sigma f)=e(g,f)\res(\alpha_{n+1}|-W_{Q_f}). \]
\end{lemma}
\proof  Theorem  \ref{strongstab} on strong stability implies that $\Tp(\sigma g)|_{\sigma f}=Tp(g)|_{f}$. Using the Interpolation Theorem \ref{interpol} we get
\[e(\sigma g,\sigma f)= \frac{\res(W_{\sigma f}|-W_{Q_{\sigma f}})}{\Tp(\sigma g)|_{\sigma f}}=
                        \frac{\res(W_f|-W_{Q_f})\res(\alpha_{n+1}|-W_{Q_f})}{\Tp(g)|_f}, \]
by noticing that $W_{\sigma f}=W_f\cup \{\alpha_{n+1}\}$ and $Q_{\sigma f}=Q_f$. \qed

 Now we sketch a geometric proof, based on a suggestion of M. Kazarian:
Let $V<J^k(n)$ be a complementary invariant subspace to $I_f$ (take the span of monomials not in $I_f$). Then for any $v\in V$ the jet
\[g_v(x_1,\dotsc,x_{n+1})=\big(g_1(x_1,\dotsc,x_{n}),\dotsc,g_p(x_1,\dotsc,x_{n}),x_{n+1}+v\big)\]
is contact equivalent to $\sigma g$. By checking the dimension we can see that in the affine neighbourhood $U\iso \Hom(I_{\sigma f},V)$ of $I_{\sigma f}\in \gr^\mu$ defined by the decomposition $J^k(n+1)=I_{\sigma f}\oplus V$ we have
\[\mathcal{R}(n+1)\sigma g\iso \mathcal{R}(n)g\times \Hom(\C x_{n+1},V).\]
This local product structure immediately implies the Tangent Lemma. \qed



Using the Localization Formula it is easy to see that the Tangent Lemma is equivalent to Theorem  \ref{strongstab}, so the second proof of the Tangent Lemma gives a direct proof of the strong stability.

\begin{example}{\bf The Thom polynomial of $A_3$}: The first case not covered in Section \ref{firstex} is the Thom series of the Morin singularity $A_3$, the contact class corresponding to the algebra $\C[[x]]/(x^4)$. Since $\mu(g)=3$ it is enough to write down the Localization Formula for $n=2$ (see Remark \ref{m-1-eleg}). The monomial ideals for $n=2$ can be identified with partitions of $\mu+1=4$: $(4), (31), (211), (22)$. These monomial ideals will be denoted by
$I_{4}$, $I_{31}$, $I_{211}$ and $I_{22}$. Germs with these monomial ideals will be denoted by $f_{4}$, $f_{31}$, $f_{211}$ and $f_{22}$.

Since $f_{4}$ is the suspension (c.f. Section \ref{sec:p-n}) of the jet $x_1\mapsto x_1^4\in J^4(1,1)$, we can apply the Tangent Lemma~\ref{tangentlemma}.
The ideal $(x_1^4)$ of $J^k(1)$ is a fixed point of the $\mathcal R$-action, so $e(x_1^4,x_1^4)=1$, therefore
\[e(f_4,f_4)=\res(\alpha_1,2\alpha_1,3\alpha_1|\alpha_2).\]

For $I_{22}=(x_1^2,x_2^2)$ we use the Interpolation Formula. The ideal $I_{22}$ has 2 generators so we need that
$\tp_{A_3}(0)=c_1^3+3c_1c_2+2c_3$. We write this polynomial in the Chern roots $\alpha_1,\alpha_2,\beta_1,\beta_2$
and restrict to $f_{22}$ ($\beta_1\mapsto 2\alpha_1,\ \beta_2\mapsto 2\alpha_2$) i.e. make the substitutions
\[c_1=2\alpha_1+2\alpha_2-\alpha_1-\alpha_2=\alpha_1+\alpha_2, \ \
c_2=\alpha_1\alpha_2-\alpha_1^2-\alpha_2^2, \ \
c_3=(\alpha_1+\alpha_2)(\alpha_2-\alpha_1)^2,\]
and we get that
\[\Tp_{A_3}(2,2)|_{f_{22}}=(\alpha_1+\alpha_2)\alpha_1\alpha_2.\]
The interpolation theorem implies that
\[e(f_4,f_{22})= \frac{\res(\alpha_1,\alpha_2,\alpha_1+\alpha_2|2\alpha_1,2\alpha_2)} {(\alpha_1+\alpha_2)
     \alpha_1\alpha_2}= \frac{(\alpha_2-\alpha_1)^2(2\alpha_1-\alpha_2)(\alpha_1-2\alpha_2)}{\alpha_1+\alpha_2}.  \]
 We also need to calculate the Euler class at $f_{31}=(x_1^3,x_1x_2,x_2)$ (the Euler class at $f_{211}$ can be obtained by permuting $\alpha_1$ and $\alpha_2$). The ideal of $I_{31}$ has three generators so we need $\tp_{A_3}(1)$ to apply the interpolation formula. This calculation is better to do with computer, the result can be found in Section \ref{sec:smallmu} at $\mu=3$.
   \end{example}

This method of calculating all the ingredients of the Localization Formula from a concrete Thom polynomial $\tp_Q(l)$ (for an appropriate $l$), will be called the Extrapolation method. Now we estimate the value of $l$ for which this method works.

\begin{proposition}\label{pmax} Let $I$ be a monomial ideal in $J^k(n)$ and assume that $a(Q)$---the minimal number of generators of $Q=J^k(n)/I$---is at most $\mu(Q)-1$. Then $b(Q)$---the minimal number of relations of $Q$---is at most $\binom{\mu(Q)}2$. \end{proposition}

\begin{proof} If $a(Q)=\mu(Q)-1$, then $I$ contains all but one quadratic monomials. They are all generators and, if the missing monomial is of the form $x_i^2$, there can be one extra generator, all together maximum $\binom{\mu(Q)}2$ generators. If $a(Q)<\mu(Q)-1$, then ``cut off'' a maximal degree monomial from $Q$: let us call the resulting algebra $Q'$. Then $\mu(Q')=\mu(Q)-1$ and $a(Q')\leq a(Q)<\mu(Q)-1$ so by an induction argument we can assume that $b(Q')\leq \binom{\mu(Q)-1}2$. Hence $b(Q)\leq b(Q')+(\mu(Q)-2)< \binom{\mu(Q)}2$.
\end{proof}

\begin{corollary}To calculate the Thom series of the nilpotent algebra $Q$ with the Extrapolation method, it is enough to know $\tp_Q(\binom{\mu(Q)-1}2)$. \qed
\end{corollary}

We saw already that the Thom series depends only on the finite data of tangent Euler classes at monomial ideals for $n=\mu(Q)-1$. Now we see that the same information is stored in the polynomial $\tp_Q(\binom{\mu(Q)-1}2)$ in a compact way.

This argument shows that it is theoretically possible to calculate a closed formula for {\em any} Thom series, as we have an algorithm based on Groebner degeneration to calculate any Thom polynomial. However this algorithm is extremely ineffective for explicit calculations (works only for trivial cases).

There is a remarkable relation among the tangent Euler classes $e(Q,I)$ for different monomial ideals $I$. Consider the Berline-Vergne-Atiyah-Bott equivariant localization formula \ref{ab} for $[\eta_Q]\in H^*_T(\gr^\mu)$:
\begin{equation}\label{eq:reciprocity}\pi_*[\eta_Q]=\sum_I\frac1{e(Q,I)},\end{equation}
where $\pi:\gr^\mu\to*$ is the collapse map. The Gysin map $\pi_*$ decreases the degree by the dimension of $\gr^\mu$, so the left hand side of (\ref{eq:reciprocity}) is 0 unless $\eta_Q$ is zero-dimensional, i.e. a fixed point of $\mathcal{R}(n)$, when (\ref{eq:reciprocity}) reduces to a tautology. These cases were treated in Section \ref{firstex}.
 As a consequence, the value of $e(Q,\M_{\mu(Q)}^2)$ can be calculated algebraically from the other Euler classes.


\subsection{Thom polynomials corresponding to algebras of small dimension}\label{sec:smallmu}

In what follows the maximal ideal $\M_i$ of $\C[[x_1,\dotsc,x_i]]$ will be considered to be a subset of $\M_{i+1}$. If $I\subset \M_i$ is an ideal we define its descendant in $\M_{i+1}$ as $I+(x_{i+1})$. Descendants of descendants are also called descendants. Observe that the factor ring of $\M_i$ by $I$ is isomorphic to the factor ring of $\M_{i+1}$ by the descendant of $I$, in particular $\codim(I\subset \M_i)=\codim(I'\subset \M_{i+1})$ for the descendant $I'$ of $I$.

Recall also that we consider the right group $\mathcal{R}(i)$ acting on $\M_i$, in particular the symmetric group $S_i\subset \GL(i)\subset \mathcal{R}(i)$ also acts on $\M_i$.

\smallskip

Let us fix $\mu\geq 1$. Consider a list of monomial ideals $I_i\subset \M_{n(i)}$ ($n(i)\leq \mu$) such that
\begin{itemize}
\item{} $\codim ( I_i\subset \M_{n(i)})=\mu$;
\item{} no ideal in the $\GL(n(i))$-orbit of $I_i$ is the descendant of an ideal in $\M_{n(i)-1}$;
\item{} the $S_\mu$-orbits of the descendants of $I_i$'s in $\M_\mu$ form a no-repetition, complete list of the codimension $\mu$ monomial ideals of $\M_\mu$.
\end{itemize}

\begin{example}\label{listofIi}
Here are examples for small $\mu$, with the notation $x,y,z,\dots=x_1,x_2,x_3,\dots$.

\noindent{$\mu=1$:} $I_1=(x^2)\subset \M_1$.

\noindent{$\mu=2$:} $I_1=(x^3)\subset \M_1$, $I_2=(x^2,xy,y^2)=\M_2^2\subset \M_2$.

\noindent{$\mu=3$:} $I_1=(x^4)\subset \M_1$, $I_2=(x^2,y^2)\subset \M_2$, $I_3=(x^2,xy,y^3)\subset \M_2$,

\noindent{\ \ \ \ \ \ \ \ \ } $I_4=\M_3^2=(x^2,y^2,z^2,xy,yz,zx)\subset \M_3$.

\noindent{$\mu=4$:} $I_1=(x^5)\subset \M_1$, $I_2=(x^2,xy,y^4)\subset \M_2$, $I_3=(x^3,xy,y^3)\subset \M_2$,
$I_4=(x^2,xy^2,y^3)\subset \M_2$,

\noindent{\ \ \ \ \ \ \ \ \ } $I_5=(x^2,y^2,z^3,xy,yz,zx)\subset \M_3$, $I_6=(x^2,y^2,z^2,xy,xz)\subset \M_3$,

\noindent{\ \ \ \ \ \ \ \ \ } $I_7=\M_4^2\subset \M_4$.
\end{example}

\begin{remark} Monomial ideals $I$ of $\M_n$ can be visualized by the set $\{(i_1,i_2,\ldots,i_n)\in \Z^n: \prod_{j=1}^n x_j^{i_j}\not\in I\}$. This set can be viewed as the $n$-dimensional generalization of (two dimensional) Young diagrams of partitions. In this language, the list $I_i$ for a given $\mu$ is the list of all ``shapes'' of cardinality $\mu+1$ Young diagrams of dimension at most $\mu$.
\end{remark}

The Localization Formula (\ref{eq:lfeu}) can now be rephrased as follows.

\begin{theorem} \label{small_mu_th} Let $\mu$ be a positive integer, and $I_i$ be a list of monomial ideals described above. Then for a nilpotent algebra of dimension $\mu$ we have
\begin{equation} \label{3rdform}\Tp_Q(n,p)=\sum_i \Sym_{I_i} \frac{ \res(\beta_1,\ldots,\beta_p|
         -W_{Q_{I_i}})}{e(Q,I_i)\cdot \res(\alpha_{n(i)+1},\ldots,\alpha_n|-W_{Q_{I_i}})}, \end{equation}
where $e(Q,I_i)$ is the virtual tangent Euler class of the closure of the set
   $$\{I\triangleleft \gr^\mu(\M_{n(i)}): \M_{n(i)}/I\cong Q\}$$
at the point $I_i$. The symmetrizer operator acts on a polynomial $p$ by
  $$\Sym_{I_i}\big(p(\alpha_1,\ldots,\alpha_n)\big)=\frac{1}{|\{\sigma\in S_n: \sigma(I_i)=I_i\}|}
    \sum_{\sigma\in S_n} p(\alpha_{\sigma(1)},\ldots,\alpha_{\sigma(n)}).$$
If $n(i)>n$, or $e(Q,I_i)=\infty$ for some $i$, then the $i$'th term in the sum in (\ref{3rdform}) is defined to be 0.
\end{theorem} \qed

\begin{corollary}
The {\em finitely many} rational functions $e(Q,I_i)$ determine the Thom polynomials $\Tp_Q(n,p)$ of the nilpotent algebra $Q$ {\em for all} $n$ and $p$.
\end{corollary}

Using the convention  $x,y,z,\ldots=x_1,x_2,x_3,\ldots$, and the following names of nilpotent algebras:
$$A_i=\M_1/(x^{i+1}),\qquad I_{a,b}=\M_2/(xy,x^a+y^b),$$
$$III_{a,b}=\M_2/(x^a,xy,y^b), \qquad \Sigma^{2,1}=\M_2/(x^2,xy^2,y^3),$$
here is a list of some Euler classes:

\bigskip

\noindent{\bf $\mu=1$:}\begin{center} $e(A_1,(x^2))=1$ \end{center}


\noindent{\bf $\mu=2$:}

\begin{center}
\[\begin{array}{|c||c|c|}
\hline I=   & (x^3)        & (x^2,xy,y^2)         \\
\hline \hline e(A_2,I)=      & 1&
\frac{1}{3}(\alpha_1-2\alpha_2)(\alpha_2-2\alpha_1) \\
\hline
\end{array}\]
\end{center}

\noindent{\bf $\mu=3$:}

\begin{center}
\[\begin{array}{|c||c|c|c|}
\hline I=   & (x^4)        & (x^2,y^2)   & (x^2,xy,y^3)       \\
\hline \hline e(A_3,I)=      & 1
                &\displaystyle{\frac{(\alpha_1-\alpha_2)^2(2\alpha_1-\alpha_2)(\alpha_1-2\alpha_2)}{(\alpha_1+\alpha_2)}}
                & \frac12 (3\alpha_2-\alpha_1)(\alpha_1-\alpha_2)^2\\
\hline e(I_{2,2},I)=  & \infty& -(\alpha_1-\alpha_2)^2&2(\alpha_1-\alpha_2)^2 \\
\hline e(III_{2,3},I)=& \infty & \infty&  \alpha_1-\alpha_2 \\
\hline
\end{array}\]
\end{center}

\noindent{\bf $\mu=4$:}

\begin{tabular}{ll}
$e(A_4,(x^5))= $& $1$,\\
 $e(A_4, (x^2,xy,y^4))=$ &
$\frac{1}{5}(\alpha_1-\alpha_2)(\alpha_1-2\alpha_2)(\alpha_1-4\alpha_2)(3\alpha_2-2\alpha_1),$
\\
$e(A_4,(x^3,xy,y^3))=$ & $
\frac{1}{5}(\alpha_1-\alpha_2)^2(2\alpha_1-3\alpha_2)(3\alpha_1-2\alpha_2),$ \\
$e(A_4,(x^2,xy^2,y^3)=$ &
$\frac{2(\alpha_1-2\alpha_2)(2\alpha_1-\alpha_2)(\alpha_1-3\alpha_2)(\alpha_1-\alpha_2)^2}{5(\alpha_1+\alpha_2)},$ \\
$e(A_4,(x^2,y^2,z^2,xy,xz))=$ & $\clubsuit \cdot
\frac{(\alpha_2+\alpha_3-2\alpha_1)(\alpha_2-2\alpha_3)(\alpha_3-2\alpha_2)}{5(\alpha_2+\alpha_3)},$ \\
$e(A_4,(x^2,y^2,z^3,xy,yz,zx))=$ &
$\spadesuit\cdot\frac{-2(\alpha_1+\alpha_2-2\alpha_3)(3\alpha_3-\alpha_2)(3\alpha_3-\alpha_1)(\alpha_1-2\alpha_2)(\alpha_2-2\alpha_1)}
{5(4\alpha_1^3-9\alpha_1^2\alpha_3-5\alpha_1^2\alpha_2-4\alpha_1\alpha_3^2+15\alpha_1\alpha_2\alpha_3-5\alpha_1\alpha_2^2-4\alpha_2\alpha_3^2-
9\alpha_2^2\alpha_3+4\alpha_2^3+9\alpha_3^3)}.$
\end{tabular}

\bigskip

\begin{tabular}{ll}
$e(I_{2,3},(x^5))=$ & $\infty$, \\
$e(I_{2,3},(x^2,xy,y^4))=$& $(\alpha_1-\alpha_2)(\alpha_1-2\alpha_2)(2\alpha_1-3\alpha_2),$\\
$e(I_{2,3},(x^3,xy,y^3))=$&
$\frac{(\alpha_1-\alpha_2)^2(2\alpha_1-3\alpha_2)(3\alpha_1-2\alpha_2)}{\alpha_1+\alpha_2},$\\
$e(I_{2,3},(x^2,xy^2,y^3))=$&$(\alpha_1-\alpha_2)^2(2\alpha_2-\alpha_1),$\\
$e(I_{2,3},(x^2,y^2,z^2,xy,xz))=$&$\clubsuit \cdot
\frac{(2\alpha_3-\alpha_2)(\alpha_3-2\alpha_2)(2\alpha_1-\alpha_2-\alpha_3)}{\alpha_1\alpha_2+\alpha_1\alpha_3+4\alpha_2^2+4\alpha_3^2-16\alpha_2\alpha_3},$\\
$e(I_{2,3},(x^2,y^2,z^3,xy,yz,zx))=$\\
\end{tabular}

\ \hfill
$\spadesuit\cdot\frac{2(\alpha_1+\alpha_2-2\alpha_3)(3\alpha_3-\alpha_1)(3\alpha_3-\alpha_2)(\alpha_1-2\alpha_2)(\alpha_2-2\alpha_1)}
{4(\alpha_1^4+\alpha_2^4)-6\alpha_1^2\alpha_2^2-5\alpha_1\alpha_2(\alpha_1^2+\alpha_2^2)+\alpha_3(-25(\alpha_1^3+\alpha_2^3)+39\alpha_1\alpha_2(\alpha_1+\alpha_2))+\alpha_3^2(29(\alpha_1^2+\alpha_2^2)-59\alpha_1\alpha_2)+\alpha_3^3(\alpha_1+\alpha_2)}$.

\bigskip
\begin{tabular}{ll}
$e(III_{2,4},(x^5))= $&$ \infty$,\\
$e(III_{2,4},(x^3,xy,y^3))=$&$ \infty,$\\
$e(III_{2,4},(x^2,xy,y^4))=$&$(\alpha_1-\alpha_2)(\alpha_1-2\alpha_2)$,\\
$e(III_{2,4},(x^2,xy^2,y^3))=$&$(\alpha_1-\alpha_2)(2\alpha_2-\alpha_1),$\\
$e(III_{2,4},(x^2,y^2,z^2,xy,xz))=$&$\clubsuit \cdot
\frac{(\alpha_2-2\alpha_3)(\alpha_3-2\alpha_2)}{\alpha_1\alpha_2+\alpha_1\alpha_3-4\alpha_2^2-4\alpha_3^2+4\alpha_2\alpha_3},$\\
$e(III_{2,4},(x^2,y^2,z^3,xy,yz,zx))=$&
$\spadesuit\cdot\frac{-(\alpha_1-3\alpha_3)(\alpha_2-3\alpha_3)}
{2(\alpha_1^2+\alpha_2^2+\alpha_3^2-2\alpha_1\alpha_3-2\alpha_2\alpha_3-\alpha_1\alpha_2)}.$
\end{tabular}

\bigskip

\begin{tabular}{ll}
$e(III_{3,3},(x^5))=$&$\infty$,\\
$e(III_{3,3},(x^2,xy,y^4))=$&$\infty,$\\
$e(III_{3,3},(x^3,xy,y^3))=$&$-(\alpha_1-\alpha_2)^2$,\\
$e(III_{3,3},(x^2,xy^2,y^3))=$&$2(\alpha_1-\alpha_2)^2,$\\
$e(III_{3,3},(x^2,y^2,z^2,xy,xz))=$&$\clubsuit \cdot(-1),$\\
$e(III_{3,3},(x^2,y^2,z^3,xy,yz,zx))=$&$
\spadesuit\cdot\frac{2(\alpha_1-2\alpha_2)(\alpha_2-2\alpha_1)}
{2\alpha_1^2+2\alpha_2^2-3\alpha_3^2-2\alpha_1\alpha_2+\alpha_1\alpha_3+\alpha_2\alpha_3}.$
\end{tabular}

\bigskip

\begin{tabular}{ll}
$e(\Sigma^{2,1},(x^5))=$&$ \infty$,\\
$e(\Sigma^{2,1},(x^2,xy,y^4))=$&$ \infty$,\\
$e(\Sigma^{2,1},(x^3,xy,y^3))=$&$\infty$,\\
$e(\Sigma^{2,1},(x^2,xy^2,y^3))=$&$\alpha_1-\alpha_2$,\\
$e(\Sigma^{2,1},(x^2,y^2,z^2,xy,xz))=$&$\clubsuit \cdot
\frac{1}{2(\alpha_1-\alpha_2-\alpha_3)},$\\
$e(\Sigma^{2,1},(x^2,y^2,z^3,xy,yz,zx))=$&$ \spadesuit\cdot
\frac{1}{\alpha_1+\alpha_2-\alpha_3},$
\end{tabular}

where

$$\clubsuit=(\alpha_1-\alpha_2)(\alpha_1-\alpha_3)(\alpha_1-2\alpha_2)(\alpha_1-2\alpha_3)(\alpha_2-\alpha_3)^2,$$
$$\spadesuit=(\alpha_1-\alpha_3)^2(\alpha_2-\alpha_3)^2(\alpha_1-\alpha_2-\alpha_3)(\alpha_2-\alpha_1-\alpha_3).$$

Theorem \ref{small_mu_th}, and the list of $e(Q,I)$-classes above give the Thom polynomial of all singularities whose associated algebra has dimension at most 4, with the following exceptions:

\begin{itemize}
\item{} We did not include $e(Q,I)$ classes for $Q=\M_i/\M_i^2=\Sigma^i$ for $i=2,3,4$, since those Thom polynomials (Giambelli-Thom-Porteous formulae) are known, see (\ref{eq:gia}).
\item{} For $\mu=3$ and $\mu=4$ we did not include the classes $e(Q,\M_{\mu}^2)$, since they can be calculated using (\ref{eq:reciprocity}).
\item{} There are three other algebras with $\mu=4$, namely
$$\M_3/(x^2,y^2,z^3,xy,yz,zx), \M_3/(x^2,y^2,z^2,xz,yz),\M_3/(xy,xz,yz,x^2-y^2,x^2-z^2).$$
Their Thom polynomials will be studied in Section \ref{sec:phi} (under the names $\Phi_{3,2}$, $\Phi_{3,1}$, $\Phi_{3,0}$).
\end{itemize}

\section{Returning to geometry} \label{sec:return}

In this section we use some simple geometric observations to calculate the Localization Formula for some singularities. We believe that this is just the beginning. In a similar fashion one can transform the deep geometric knowledge of singularity theorists into further formulas.
 \subsection{The punctual Hilbert scheme} The localization Formula reduces the Thom polynomial calculations to the study of the space of ideals $\mathcal{H}^m(n)=\mathcal{H}^m(n,k)$ in $\gr^m(J^k(n))$. We suppress $k$ from the notation since these spaces are isomorphic for all $k\geq m$, what we will assume. The variety $\mathcal{H}^m(n)$ was studied under the name {\em local punctual Hilbert scheme} (with the reduced scheme structure) by A. Iarrobino, J. Damon, A. Galligo, T. Gaffney and others (see \cite{iarrobino, damon-galligo, gaffney}). The connection of this Hilbert scheme with singularity theory is well known, however the $\mathcal{R}(n)$-equivariant theory needed for Thom polynomial calculations is not developed yet. We do not pursue this approach here but we believe that it would lead to strong results.

The structure of $\mathcal{H}^m(n)$ is complicated. In general even its dimension is not known. It has many components, one of which is  the closure of the orbit $\mathcal{R}(n)(x_1^{m+1},x_2,\dotsc,x_n)$. For $n=2$ this is the only component, for $n\geq 3$, there can be others. Nevertheless we can see that the calculation of the Thom polynomials of the Morin singularities $A_m$ is related to the study of the singularities of this component at monomial ideals.

The  Hilbert scheme $\mathcal{H}^m(n)$ has many {\em smooth} $\mathcal{R}(n)$-invariant subvarieties. For the corresponding contact classes we can easily calculate the Localization Formula.

\subsection{Subgrassmannians}\label{subgrass} Let $V$ be a $d$-dimensional subspace of $P:=\Hom(\sym^{k+1}\C^n,\C)$ for $d<\dim(P)=\binom{n+k-1}k$.  Let $N\geq 1$ and consider the ideal $I_V<J^{k+N}(n)$ generated by $V$ and $\Hom(\sym^{k+2}\C^n,\C)$. We have the $\GL(n)$-equivariant embedding $j:V\mapsto I_V$ mapping the Grassmannian $\gr_d(P)$ into the  Hilbert scheme $\mathcal{H}^m(n)$ where $m=\sum_{i=2}^{k+1}\binom{n+i-1}i-d$. These subgrassmannians were used by Iarrobino in \cite{iarrobino:dim} to give a lower bound on the dimension of $\mathcal{H}^m(n)$. The corresponding contact classes
\[\Sigma^{n^k(d)}(n,p)=\Sigma^{\overbrace{n,\dotsc,n}^k(d)}(n,p):=\{g\in J^{k+N}(n,p):I_g\in j(\gr_d(P))\}\]
were studied by J. Damon in \cite{damonphd}. He also calculated the Thom polynomial of some of these classes. The Localization Formula gives the answer for all of these cases:

Let W be the set of integer $k+1$-tuples $w=(w_1,\dotsc,w_{k+1})$ with $1\leq w_1\leq w_2\leq\cdots\leq w_{k+1}\leq n$ and let $\alpha_{w}$ denote the weight $\sum w_i\alpha_i$. Then $\{\alpha_w:w\in W\}$ is the set of weights of $P=\Hom(\sym^{k+1}\C^n,\C)$. The $T(n)$-fixed points of the Grassmannian $\gr_d(P)$ are identified with the $d$-element subsets of $W$. For a fixed point $S$ let
\[ [E_S]=\prod_{i=1}^p\prod_{\sigma\in S}(\beta_i-\alpha_\sigma) \ \ \text{and} \ \
e_S=e(T_S \gr_d(P))=\prod_{\sigma\in S}\prod_{\sigma'\in W\setminus S}(\alpha_\sigma-\alpha_{\sigma'}).\]
Then by the Localization Formula we get
\begin{equation}[\Sigma^{n^k(d)}(n,p)]=[\Sigma^{n^k}(n,p)]\cdot\sum_{S}\frac{[E_S]}{e_S}.\label{eq:gr} \end{equation}
Notice that the result is independent of $N$ in accordance with Theorem \ref{k}. For the choice $N=1$ we get $I_V=V$.

Recall that we calculated $[\Sigma^{n^k}(n,p)]$ in Section \ref{firstex}. If $k=1$ and $d=\binom{n+1}2-1$ then by Proposition \ref{unique} we can calculate the stable Thom polynomial from (\ref{eq:gr}). In fact these Thom polynomials can be calculated by a direct geometric argument (see the proof of Theorem \ref{Phi_quotient}). (Knowing the result from the Localization Formula certainly helped to find the  geometric argument.)

  \section{Thom series of $\Phi_{m,r}$ singularities} \label{sec:phi}

The subgrassmannian $j\gr^1(\Hom(\sym^2\C^m,\C))\iso\P( \sym^2\C^m)$ splits into orbits $X(m,r)$ according to the corank of the symmetric matrices in $\sym^2\C^m$. The orbits correspond to the following algebras.

\begin{definition} Let $m>r$ be nonnegative integers. The quotient of $J^2(m)$ by the ideal
 \[J_{m,r}=\left\{ \sum_{1\leq i\leq j\leq m} a_{ij} x_ix_j\ :\ \sum_{r+1}^m a_{ii}=0 \right\}\]
will be denoted by $\Phi_{m,r}$.
\end{definition}

\noindent A finite generator set of $J_{m,r}$ (as an ideal but also as a vector space) is given by:
  \[J_{m,r}=\left\langle x_ix_j, x_k^2, x_{r+1}^2-x_l^2: 1\leq i<j\leq  m, 1\leq k\leq r, r+2\leq l\leq m \right\rangle\] for $r<m-1$, and
  \[J_{m,m-1}=\left\langle x_ix_j, x_k^2\ :\ 1\leq i<j\leq m, 1\leq k \leq m-1\right\rangle.\]
Observe that for small values of the parameters $m,r$ we recover familiar algebras:
      \[ \Phi_{1,0}=A_2, \qquad \Phi_{2,0}=I_{2,2}, \qquad \Phi_{2,1}=III_{2,3}.\]
Following our previous convention, the singularities corresponding to the algebras $\Phi_{m,r}$ in $\E_0(n,n+l)$ will be denoted by $\Phi_{m,r}(n,n+l)$. Calculation shows that
   \[\codim \big( \Phi_{m,r}(n,n+l) \subset \E_0(n,n+l) \big)= (m+1)l+ \left( \binom{m+1}{2}+\binom{r+1}{2}+1\right).\]

\subsection{Thom polynomials of $\Phi_{n,r}$ in terms of Chern roots.}

To calculate the Thom polynomials of these classes---using the Localization Formula---we need the $\GL(n)$-equivariant cohomology classes $[X(n,r)\subset \P(\sym^2\C^n)]$ restricted to the $T(n)$-fixed points of $\P(\sym^2\C^n)$. The cohomology class of the cone of $X(n,r)$ was
calculated in \cite{harris-tu} and \cite{jlt}:
    \[ [\text{Cone}X(n,r)\subset \sym^2\C^n] =2^r\Delta_{r,r-1,\dotsc,2,1},\]
where $\Delta_{r,r-1,\dotsc,2,1}=\det(c_{r+1+j-2i})_{i,j=1,\dots,r}$ denotes the Schur  polynomial in the Chern classes $c_1,\dotsc,c_n$, corresponding to the partition $(r,r-1,\dotsc,2,1)$. Using \cite[\S6]{forms} we can calculate the $T(n)$-equivariant {\em projective Thom polynomial}
\[[X(n,r)\subset\P(\sym^2\C^n)]=2^r\Delta_{r,r-1,\dots,2,1}(\alpha_1-\frac12\xi,\dotsc,\alpha_n-\frac12\xi)\ \ \text{in}\]    \[  H^*_{T(n)}(\P( \sym^2\C^n))\iso \Z[\alpha_1,\dotsc,\alpha_n,\xi]/\prod(\alpha_i+\xi).\]
We need the restriction of this class to the fixed points $\{f_{ij}:1\leq i\leq j\leq n\}$ of $\P( \sym^2\C^n)$:
\begin{equation}[X(n,r)]|_{f_{ij}}=2^r\Delta_{r,r-1,\dotsc,2,1}(\alpha_1-\frac12(\alpha_i+\alpha_j),\dotsc,\alpha_n-\frac12(\alpha_i+\alpha_j)).\label{eq:xnr}
\end{equation}
The other components of the Localization Formula are
\[[E_{ij}]=\res(\beta_1,\dotsc,\beta_p|\alpha_i+\alpha_j), \ \
e_{ij}^{(n)}=\res(K_{ij}^{(n)}|\alpha_i+\alpha_j), \ \text{for} \ K_{ij}^{(n)}=\{\alpha_k+\alpha_l:k\leq l,\ (k,l)\neq (i,j)\}.\]
Hence the Interpolation formula yields to the following
\begin{theorem} \label{tpgnr} The Thom polynomial of  $\Phi_{n,r}$ is
    \begin{equation} \label{phi_local_exp}\Tp_{\Phi_{n,r}}(n,p)=\res(\beta_1,\dotsc,\beta_p|\alpha_1,\dotsc,\alpha_n)
                       \sum_{1\leq i\leq j\leq n}\frac{[E_{ij}]}{e_{ij}^{(n)}}[X(n,r)]|_{f_{ij}}. \end{equation}\qed
\end{theorem}

\subsection{Thom polynomials of $\Phi_{m,r}$ in terms of quotient Chern classes}\label{sec:quotientchern}\mbox{}
Since $\mu(\Phi_{n,r})=n+1$, the polynomial $\Tp_{\Phi_{n,r}}(n,p)$ determines the Thom series of $\Phi_{n,r}$ by Proposition \ref{unique}. We devote this section to the calculation of these Thom series.

\subsubsection{Notations from algebraic combinatorics}Let
\[A_n=\{\alpha_1,\ldots,\alpha_n\},\qquad B_p=\{\beta_1,\ldots,\beta_p\}.\]
For a partition $\lambda=(\lambda_1,\ldots,\lambda_s)$ and variables $x_1,\ldots, x_t$ we define
     \[  \Delta_\lambda(x_1,\ldots,x_t)=\det(\sigma_{\lambda_i+j-i})_{1\leq i,j\leq s},\]
where $\sigma_i=\sigma_i(x_1,\ldots,x_t)$ is the $i$th elementary symmetric polynomial of $x_1,\ldots, x_t$ (i.e.
$\sigma_1=\sum x_i$, $\sigma_2=\sum_{i\not= j} x_ix_j$ etc). The symbol $\Delta_\lambda$ without arguments (as before, in (\ref{schurdef})) will denote the determinant
    \[\Delta_\lambda=\det(c_{\lambda_i+j-i})_{1\leq i,j\leq s}\]
with entries the quotient variables (\ref{quotient_vars}), (\ref{eq:c}). We will use partitions, and their notations as are in \cite{fulton:young}. For example, $\overline{\lambda}$ will denote the {\sl conjugate} partition of $\lambda$. Addition of partitions is defined coordinatewise, $a^b$ means $b$ copies of $a$, and concatenation is indicated by a comma. For example $(3^4+(2,1),(1,1))=(5,4,3,3,1,1)$. We will need the staircase partition
     \[\rho_s=(s,s-1,\ldots,2,1).\]

\begin{definition} \label{def:segre} The Schur coefficients of the equivariant Segre classes of $\Sym^2\C^n$ will be denoted by double brackets; namely:
   \[ \frac{1}{\prod_{1\leq i\leq j \leq n} (1-\alpha_i-\alpha_j)}=
           \sum_{I} ((I))  \Delta_{\overline{I-\rho_{n-1}}}(\alpha_1,\ldots,\alpha_n).\]
Here $I$ runs through all length $n$ sequences $I=(i_1,i_2,\ldots, i_n)$ with $i_1>i_2>\ldots >i_n\geq 0$.
\end{definition}

The numbers $((I))$ are positive; their combinatorics, as well as recursion and Pfaffian formulas are studied in \cite{pragacz:enumgeo,lalat,pragacz_dd}. For practical purposes, the following recursion is most useful
  \[ r ((i_1,\ldots, i_r)) -2\sum_{k=1}^r ((i_1,\ldots, i_{k-1}, i_k-1, i_{k+1},\ldots, i_r))= \begin{cases} 0 & i_r>0 \\
               ((i_1,\ldots,i_{r-1})) & i_r=0,\end{cases}\]
together with the conventions $((0))=1$, and that $((I))=0$ if $I$ does not satisfy $i_1>i_2>\ldots >i_n\geq 0$. For example $((i))=2^i$, $((i,0))=2^i-1$, $((2,1))=3$, $((3,1))=10$.

Now we are ready to present the Thom series corresponding to the algebras $\Phi_{m,m-s}$.

\begin{theorem}\label{Phi_quotient}We have
\begin{equation}\label{Phi_quotient_formula} \tp_{\Phi_{m,m-s}}(l)= \sum_I ((I))\Delta_{I'}, \end{equation}
where
\begin{eqnarray*}
I' & = & ((l+s)^s+I-\rho_{s-1}, (l+m)^{m-s},l+m+1-s-|I|) \\
   & = & (l+1+i_1, l+2+i_2, \ldots, l+s+i_s, {\underbrace{{l+m,\ldots,l+m}}_{\mbox{$m-s$}}}, l+m+1-s-|I|),
\end{eqnarray*}
and the summation is for sequences $I=(i_1,\ldots,i_s)$ with $i_1>i_2>\ldots>i_s\geq m-s$ and $|I|=i_1+\ldots+i_s\leq l+m-s+1$.
\end{theorem}
The special case $m=s=2$ (the singularity $I_{2,2}$) was proved in \cite{pragacz:i22}.

\begin{remark} We can formally change the summation for all sequences $I$ with $i_1>i_2>\ldots>i_s\geq 0$ without changing the sum. Indeed, if $|I|$ is larger than $l+m-s+1$ then the $\Delta$ polynomial is zero, since the last part in the partition is negative. If $i_s<m-s$ then in the determinant expansion of $\Delta$ the $s$th row coincides with one of the next $m-s$ rows, hence, again $\Delta=0$.
\end{remark}

\begin{remark} The sum of the parts of all the partitions $I'$ above is $l(m+1)+\binom{m+1}{2}+\binom{m-s+1}{2}+1$, consistent with the fact that this is the codimension of the singularity $\Phi_{m,m-s}(*,*+l)$ in $\E_0(*,*+l)$.
\end{remark}

\begin{remark} The formula gets particularly simple if $r=m-1$:
  \begin{equation} \label{eq:phinn-1} \tp_{\Phi_{n,n-1}}(l)=2^{n-1}\sum_{i=0}^{l+1}2^i
       \Delta_{n+l+i,{\underbrace{\scriptstyle{n+l,\dotsc,n+l}}_{\mbox{$\scriptstyle{n-1}$}}},l+1-i}, \end{equation}
which recovers the Ronga-formula for $A_2=\Phi_{1,0}$ and gives the Thom series of $III_{2,3}=\Phi_{2,1}$ for n=3. The Thom series of $III_{2,3}$ was also calculated recently in \cite{ozturk:3}.
\end{remark}

\begin{proof} First we prove Theorem \ref{Phi_quotient} for the case $s=m$. This is a direct geometric argument, not using localization.
We want to calculate the equivariant Poincar\'e dual of the $\Phi_{m,0}$-jets $X_{m,0}\subset\Hom(\C^m,\C^p)\oplus\Hom(\sym^2\C^m,\C^p)$. By definition:
\[ [X_{m,0}]=e\big(\Hom(\C^m,\C^p)\big)\cdot[\Sigma^1(\sym^2\C^m,\C^p)],  \]
where $\Sigma^1(V,W)$ denotes the corank 1 linear maps from $V$ to $W$. We have that
\[ e\big(\Hom(\C^m,\C^p)\big)=\res(B_p|A_m), \ \text{and}\ [\Sigma^1(\sym^2\C^m,\C^p)]=c_q(\C^p\ominus\sym^2\C^m), \]
where $q=p-\binom{m+1}2+1$ and $c_q(\C^p\ominus\sym^2\C^m)$ denotes the $q$\textsuperscript{th} (equivariant) Chern class of the formal difference $\C^p\ominus\sym^2\C^m$. The second statement is the Giambelli-Thom-Porteous formula. Now
\[ c_q((\sym^2\C^m,\C^p))=\sum_{i=0}^qc_{q-i}(\C^p)s_i(\sym^2\C^m), \]
where $s_i$ denote the Segre classes, which are defined by the identity
\[ (1+c_1t+c_2t^2+c_3t^3+\cdots)\cdot(1-s_1t+s_2t^2-s_3t^3+-\cdots)=1. \]
As we mentioned in Definition \ref{def:segre} the Segre classes of $\sym^2\C^m$ can be expressed from the Chern roots of $\C^m$:
\[ s_i(\sym^2\C^m)=\sum_{I} ((I))  \Delta_{\overline{I-\rho_{m-1}}}(\alpha_1,\ldots,\alpha_m),\]
where $I$ runs through all length $m$ sequences $I=(i_1,i_2,\ldots, i_m)$ with $i_1>i_2>\ldots >i_n\geq0$ with $|I|-\binom{m}2=i$. We finish the proof of the formula (\ref{Phi_quotient_formula}) for $s=m$ by recalling the Factorization Property of Schur polynomials.

\begin{lemma}\label{factorization_property} \cite[I.3]{macdonald} (Factorization Formula) Let $n,p$ be nonnegative integers, and let the quotient Chern
classes be defined as in (\ref{quotient_vars}). Suppose that $(p^n+\lambda,\mu)$ is a partition. Then
   \[ \Delta_{p^n+\lambda,\mu}=\res(A_n,B_p) \Delta_{\mu}(B_p) \Delta_{\overline{\lambda}}(A_n).\] \end{lemma}

For $m=2$ we have $\Phi_{2,0}=I_{2,2}$. The Thom series of $I_{2,2}$ was calculated by several authors \cite{dstab}, \cite{pragacz:i22}, \cite{kaza:gysin}.

\medskip

Now we go on with the proof of Theorem \ref{Phi_quotient}. An interpretation of what we proved so far
is that the expression in (\ref{phi_local_exp}) is equal to the expression in
(\ref{Phi_quotient_formula}) for $s=m$. In the remaining of the proof we will use this statement to prove that expression
(\ref{phi_local_exp}) agrees with expression (\ref{Phi_quotient_formula}) for any $m>s$. As before, $b_j$ will denote the $j$th elementary symmetric polynomial of the $\beta_i$'s.

The equality of formula (\ref{phi_local_exp}) with (\ref{Phi_quotient_formula}) for $s=m$ can be written---using the
Factorization Formula, Lemma~\ref{factorization_property}---as
\begin{equation}\label{tudjuk1}\sum_{1\leq i \leq j \leq s}
\frac{\res(B_{s+l}|\alpha_i+\alpha_j)}{\res(K^{(s)}_{i,j}|\alpha_i+\alpha_j)}=\sum_I
((I)) b_{l+1-|I|} \Delta_{\overline{I-\rho_{s-1}}}(\alpha_1,\ldots,\alpha_s).\end{equation}
What we want to prove is the equality of these two formulas for any $m$ and $s$, that is (using the Factorization Formula again)
\begin{equation}\label{akarjuk1}  \sum_{1\leq i \leq j \leq m}
    \frac{\res(B_{m+l}|\alpha_i+\alpha_j)}{\res(K^{(m)}_{i,j}|\alpha_i+\alpha_j)}2^{m-s}\Delta_{\rho_{m-s}}
           \big(\alpha_1-\frac{\alpha_i+\alpha_j}{2},\ldots,\alpha_m-\frac{\alpha_i+\alpha_j}{2}\big)=\end{equation}
  \[\sum_I ((I)) b_{l+1+m-s-|I|} \Delta_{\overline{I-(m-s)^s-\rho_{s-1}}}(\alpha_1,\ldots,\alpha_m).\]

Checking the coefficients of $b_{s+l-k}$ and $b_{m+l-k}$ respectively, in these equations we can reduce the theorem to the following problem: Knowing
\begin{equation} \label{tudjuk2} \sum_{1\leq i \leq j \leq s}
\frac{(\alpha_i+\alpha_j)^k}{\res(K^{(s)}_{i,j}|\alpha_i+\alpha_j)}=\sum_{|I|=k+1-s} ((I))
\Delta_{\overline{I-\rho_{s-1}}}(\alpha_1,\ldots,\alpha_s),\end{equation}
we want to prove
\begin{equation}\label{akarjuk2}\sum_{1\leq i \leq j \leq m}
\frac{(\alpha_i+\alpha_j)^k}{\res(K^{(m)}_{i,j}|\alpha_i+\alpha_j)}2^{m-s} \Delta_{\rho_{m-s}}
\big(\alpha_1-\frac{\alpha_i+\alpha_j}{2},\ldots,\alpha_m-\frac{\alpha_i+\alpha_j}{2}\big)=\end{equation}
  \[\sum_{|I|=k+1-s} ((I)) \Delta_{\overline{I-(m-s)^s-\rho_{s-1}}}(\alpha_1,\ldots,\alpha_m).\]

We recall the Gustafson-Milne identity:
\begin{lemma}\label{gs_lemma} \cite{gm}, \cite{cl} Let $m\geq s$ be nonnegative integers. If $H\subset \{1,\ldots,m\}$ then the set $\{\alpha_h\}_{h\in H}$ will be denoted by $\alpha_H$, and the set
$\{\alpha_1,\ldots,\alpha_m\}\setminus \alpha_H$ will be denoted by $\alpha_{\overline{H}}$. Let the partition
$\mu=(\mu_1,\mu_2,\ldots)$ satisfy $\mu_1\leq s$. Then we have
  \[ \Delta_{\mu} (\alpha_1,\ldots,\alpha_m)=\sum_{H\subset\{1,\ldots,m\}, |H|=s}
                \frac{\Delta_{s^{m-s},\mu}(\alpha_H)}{\res(\alpha_H|\alpha_{\overline{H}})}.\]
\end{lemma}

The Gustafson-Milne identity implies that the right hand side of (\ref{akarjuk2}) is obtained from the right hand side of
(\ref{tudjuk2}) by the following operation:
\[ p(\alpha_1,\ldots,\alpha_s) \mapsto \sum_{H \subset \{1,\ldots,m\}, |H|=s}
                \frac{p(\alpha_H)}{\res(\alpha_H|\alpha_{\overline{H}})}.\]
Hence it is enough to show that the same operation maps the left hand sides into each other, too. That is we need to prove
\begin{equation} \sum_{H \subset \{1,\ldots,m\}, |H|=s} \frac{\sum_{i \leq j \in H}
\frac{(\alpha_i+\alpha_j)^k}{\res(K^{H}_{i,j}|\alpha_i+\alpha_j)}} {\res(\alpha_H|\alpha_{\overline{H}})}=\end{equation}
   \[ \sum_{1\leq i \leq j \leq m}\frac{(\alpha_i+\alpha_j)^k}{\res(K^{(m)}_{i,j}|\alpha_i+\alpha_j)}2^{m-s}
      \Delta_{\rho_{m-s}}\big(\alpha_1-\frac{\alpha_i+\alpha_j}{2},\ldots,\alpha_m-\frac{\alpha_i+\alpha_j}{2}\big).\]
For this, the following Lemma will be useful.

\begin{lemma}\label{thirdlemma} Let $m>s$ be non-negative integers, and consider the variables $\gamma_1,\ldots,\gamma_m,y$. For a subset $H\subset \{1,\ldots,m\}$ we set $\overline{H}=\{1,\ldots,m\}-H$, $\gamma_H=\{\gamma_i\}_{i\in H}$, $\gamma_{\overline{H}}=\{\gamma_i\}_{i\in \overline H}$. We have
\[ \sum_{H\subset \{1,\ldots,m\}, |H|=s}\frac{\Delta_{\rho_s}(\gamma_H)}{\res(\gamma_H,\gamma_{\overline{H}})}
\prod_{i\in H}\prod_{j\in \overline{H}} (\gamma_i+\gamma_j)=\Delta_{\rho_s}(\gamma_1,\ldots,\gamma_{m},y,-y).\]
\end{lemma}

\begin{proof}In this proof we use the Thom polynomials of the representation of $\GL_m$ on $\Lambda^2\C^m$, see e.g. \cite[\S 3]{forms}.  Consider the canonical exact sequence of vector bundles $0\to S \to E \to Q \to 0$ over the Grassmannian $Gr_s\C^m$, and the diagram of maps
\[\xymatrix{\Lambda^2 Q  \ar@{^{(}->}@<-3pt>[rr]^{i} \ar@/^2pc/[rrr]^{\phi} \ar[d] &  & \Lambda^2 E
 \ar[d]\ar[r]^{\pi_2} & \Lambda^2 \C^m \\ Gr_s\C^m  \ar@{->}[rr]^{id} &  & Gr_s\C^m, & } \]
with the action of the $n$-torus. Using the fact that $\phi(\Lambda^2 Q)$ is the set of two forms of corank at least $s$,
and the identification $\Lambda^2 E =\Lambda^2 S \oplus \Lambda^2 Q\oplus (S \otimes Q)$, Proposition \ref{31} gives the statement of the lemma for $y=0$. We leave to the reader to prove the fact that the right hand side is independent of $y$. (One way of proving this is the identification of the right hand side with an {\sl incidence class} of two orbits of the representation of $\GL_{m+2}$ on $\Lambda^2 \C^{m+2}$, see \cite{zsolt}.)
\end{proof}

Lemma \ref{thirdlemma} (with the substitution $\gamma_u=\alpha_u-(\alpha_i+\alpha_j)/2$) can be used to show that
the coefficient of $(\alpha_i+\alpha_j)^k$ are the same on the two sides, for all $i\leq j$. This completes the proof of Theorem \ref{Phi_quotient}.
\end{proof}

\section{Iterated residue formulae and generating functions}\label{sec:generatingfn}

In \cite[\S 6.2]{bsz06} (see also \cite{szenes}) B\'erczi and Szenes used one rational function---we will call it {\em generating function}---and the iterated residue operation, to encode the Thom polynomial of all singularities corresponding to the same nilpotent algebra of type $A_d$. We show that generating functions can be assigned to other singularities. We give some examples and indicate how the generating function can be a useful tool in future studies of Thom polynomials.

Certain rational functions in the variables $z_1,\ldots,z_d$ generate polynomials in the quotient variables through the {\em iterated residue operation}, which we describe now, following \cite{bsz06}. Consider $\C^d$ with coordinates $z_1,\ldots,z_d$. Let $\omega_1,\ldots,\omega_N$ be linear forms on $\C^d$, and let $h(z_1,\ldots,z_d)$ be a polynomial. We define the iterated residue operator by
\begin{equation} \RES_d \frac{h(z_1,\ldots,z_d)}{\prod_{i=1}^N \omega_i}=\int_{|z_1|=R_1}
                 \cdots\int_{|z_d|=R_d}\frac{h(z_1,\ldots,z_d)dz_1\ldots dz_d}{\prod_{i=1}^N \omega_i}, \end{equation}
where $0<<R_1<< R_2<<\ldots<<R_d$. This definition makes sense up to a choice of a sign, but this will be enough for our purposes: in the formulas below we always mean the iterated residue with the appropriate choice of $\RES_d (dz_1\ldots dz_d)=\pm 1$. That is, we will describe certain expressions up to sign.

We will use the notations
$D_j=\sum_{i=0}^\infty \frac{c_i}{z_j^i}$ and $\dis_\mu=\prod_{i=1}^\mu \prod_{j=i+1}^\mu (z_i-z_j)$

The following conjecture is an extension of the Theorem (7.2) in \cite{bsz06}, where it is proved for Morin singularities.
We arrived at this conjecture while discussing the problem with M. Kazarian. He informed us that in his work in progress \cite{kaza:noas} he would prove it.
\begin{conjecture}\label{residue_conj}
Let $Q$ be a $\mu$-dimensional, commutative, nilpotent algebra with $\deg(\tp_Q(l))= \mu\cdot l +\gamma$.
\newline
\noindent{\bf (a)} There exists a rational function $k_Q$---called the generating function of $Q$---in the variables $z_1,\ldots,z_\mu$, of degree $\gamma-\binom{\mu+1}{2}$ such that
\begin{equation} \tp_Q(l)=\RES ( k_Q \cdot \dis_\mu \cdot \prod_{i=1}^\mu z_i^l D_i).\end{equation}
\newline
\noindent{\bf (b)} The generating function $k_Q$ has the form \begin{equation}k_Q(z_1,\ldots,z_\mu)=\frac{h(z_1,\ldots,z_\mu)}{\prod_{a\in A} (z_{i_a}+z_{j_a}-z_{s_a})},\label{gfn_form}\end{equation} where $h$ is a polynomial, and $\{i_a,j_a,s_a\}_{a\in A}$ is a repetition-free list of indexes with $i_a \leq j_a<s_a$ for all $a\in A$.
\end{conjecture}

The function $k_Q$ is not unique in general. The Giambelli-Thom-Porteous formula (\ref{eq:gia}) can be encoded by setting
  \[ k_{\Sigma^r}=\prod_{i=1}^{r-1}z_i^i \qquad\qquad (\text{here\ } \mu=r, \gamma=r^2). \]

Formula (7.2) of \cite{bsz06} can be interpreted as the the existence of $k_Q$ for $Q=A_i, (i=1,2,\ldots$), as well
as a concrete form of its denominator (all indices with $i_a+j_a\leq s_a$). For the $A_i$ singularity $\mu=\gamma=i$, hence the degree of $k_{A_i}$ is $-\binom{i}{2}$. B\'erczi and Szenes also calculated $k_{A_i}$ for $i=1,..,6$. Here is the first three of their results:
\[ k_{A_1}=1, \qquad k_{A_2}=\frac{1}{2z_1-z_2}, \qquad k_{A_3}=\frac{1}{(2z_1-z_2)(2z_1-z_3)(z_1+z_2-z_3)}.\]

In Section \ref{sec:phi} the singularities $\Phi_{m+r,r}$ were considered ($r=0,1,\ldots$, $m=1,2,\ldots$). For these singularities we have $\mu=m+r+1$, and $\gamma=\binom{m+r+1}{2}+\binom{r}{2}+1$, and hence $\deg k_{\Phi_{m+r,r}}=\binom{r+1}{2}-m$. The results of Section \ref{sec:quotientchern} can be summarized by the following generating functions
\[ k_{\Phi_{m+r,r}}(z_1,\ldots,z_{m+r+1})=
          \frac{ \prod_{i=1}^{r-1} z_{m+1+i}^i}{2^{m-1}(2z_1-z_{m+1})\prod_{i=1}^{m-1}(z_i+z_{i+1}-z_{m+1})}.\]
For example we have $k_{\Phi_{2,0}}=k_{I_{2,2}}=\frac1{2(2z_1-z_3)(z_1+z_2-z_3)}$ and $k_{\Phi_{2,1}}=k_{III_{2,3}}=\frac1{(2z_1-z_3)}$.
With some experimenting with the computer one can find  generating functions for the remaining nilpotent algebras with $\mu\leq4$:
\[ k_{III_{2,4}}=\frac1{(2z_1-z_2)(z_1+z_2-z_3)(2z_1-z_4)(z_1+z_2-z_4)}, \]
\[ k_{III_{3,3}}=\frac1{4(2z_1-z_3)(z_1+z_2-z_3)(2z_1-z_4)(z_1+z_2-z_4)}, \]
\[ k_{I_{2,3}}=\frac1{(2z_1-z_4)(2z_2-z_3)(2z_2-z_4)(z_1+z_2-z_4)(z_2+z_3-z_4)}, \]
\[ k_{\Sigma^{2,1}}=\frac1{(2z_1-z_3)(z_1+z_2-z_3)(2z_1-z_4)}. \]

Now we want to explore the connection of the conjectured residue-form and the localization form of Thom polynomials.
Recall that for a nilpotent algebra $Q$, and an ideal $I$, the virtual Euler class $e(Q,I)$ was defined in Remark \ref{smooth}. For a function $f$ in variables $z_1,\ldots,z_\mu$ define the asymetrization operator
\[\Asym_\mu( f ) = \sum_{\sigma\in S_\mu} \varepsilon(\sigma)f(z_{\sigma(1)},\ldots,z_{\sigma(\mu)}),\]
where $\varepsilon(\sigma)$ is the sign of the permutation $\sigma$. For a function $f$ with variables $\alpha_i$ let $f|_{\alpha_i:=z_i}$ be the same function with the variables changed to $z_i$.

\begin{conjecture} \label{res_vs_local} Suppose Conjecture \ref{residue_conj}(a) holds. Then
 \[ \Asym_\mu ( k_Q ) = \frac{\dis_\mu}{e(Q,\M_\mu^2)|_{\alpha_i:=z_i}}.\]
\end{conjecture}

If we assume part (b) of Conjecture \ref{residue_conj} as well, then Conjecture \ref{res_vs_local} reduces to a remarkable (conjectured) identity for iterated residues (or the Orlik-Salamon algebra) associated with the hyperplane arrangement $\cup_A \{z_{i_a}+z_{j_a}-z_{s_a}=0\}\cup \cup_{i,j}\{z_i=\alpha_j\}$.

\begin{remark} Conjecture \ref{res_vs_local} can be used to guess the function $k_Q$, as soon as the class $e(Q,\M_\mu^2)$ is known. In practice $e(Q,\M_\mu^2)$ can be calculated using Theorem \ref{interpol}, or equation (\ref{eq:reciprocity}). Its denominator is a symmetric function, which is a product of factors $\alpha_i+\alpha_j-\alpha_s$. Then one has finitely many choices to guess for the denominator of $k_Q$. (Eg. for $2\alpha_1-\alpha_2$ we can choose $2z_1-z_2$ or $2z_3-z_5$, etc.) Knowing the degree of $k_Q$ we also find the degree of the numerator. Putting all this together one arrives an oftentimes effective procedure to find the $k_Q$ function.
\end{remark}

\section{The Localization Formula for the ``small $p$'' case} \label{sec:smallp}\mbox{}

In Theorem \ref{lf} we gave a localization formula for  the Thom polynomial $\Tp_Q(n,p)$ where $Q$ is a  nilpotent algebra. We can evaluate this formula for any $p$, even if $Q$ cannot be defined by $p$ relations. Sometimes it is not zero and we would like to interpret these ``small $p$'' cases.

Fix $n$ and $p$ and assume that $\dim(Q)=m$. Then the correspondence variety---introduced in Definition \ref{def:correspondence}---of $Q$ is
   \[  C(Y_Q)=\{(I,g)\in\gr^m\times J^k(n,p)\ |\ I\in Y_Q, I_g\subset I\},  \]
where $Y_Q=\overline{\{I\in\gr^m: Q_I\iso Q\}}$. We have the restriction of the second projection
    \begin{equation}  \label{phi}   \phi:C(Y_Q)\to \eta_Q(n,p):=\phi(C(Y_Q)).   \end{equation}
In other words
  \[  \eta_Q(n,p)=\{g\in J^k(n,p):\exists\ I\in Y_Q \ \text{such that}\ I_g\subset I\}.   \]
We showed in Proposition \ref{birat} that if $n\geq a(Q)$ (minimal number of generators of $Q$) and $p\geq b(Q)$ (minimal number of relations of $Q$), then $\phi$ is birational. In some cases $\phi$ is birational for smaller $p$ as well. In these cases the Localization Formula still calculates $[\eta_Q(n,p)]$.

\subsection{The singularities $III_{a,b}$ and $I_{a,b}$} Consider the nilpotent algebra $III_{a,b}= \M_2/(xy,x^a,y^b)$. Germs $g:(\C^n,0) \to (\C^p,0)$ with this algebra only exist if $p\geq n+1$. Yet, consider $p=n$ and the map
$$\phi: C(Y_{III_{a,b}}) \to  \eta_{III_{a,b}}(n,n)$$
of (\ref{phi}). Clearly the $I_{a,b}$ germ $g=(xy,x^a+y^b,x_3,\dots,x_n)$ is in $\eta_{III_{a,b}}(n,n)$. Therefore $\overline{\mathcal Kg}\subset \eta_{III_{a,b}}(n,n)$. Checking their dimensions we get that in fact $\overline{\mathcal Kg}= \eta_{III_{a,b}}(n,n)$. One can verify that the only ideal in the $\mathcal{R}$-orbit of $(xy,x^a,y^b,x_3,\dots,x_n)$ containing the ideal $(xy,x^a+y^b,x_3,\dots,x_n)$ is $(xy,x^a,y^b,x_3,\dots,x_n)$ itself, therefore $\phi$  is generically one to one. It implies that (see Section \ref{sec:d-stab})
\begin{theorem} \label{thm:iabiiiab}
For $a,b\geq 2$
\[  \tp_{I_{a,b}}(0) =\tp_{III_{a,b}}(1)^{\flat(a+b-2)}.  \]
\end{theorem}

The Thom polynomials occurring in Theorem \ref{thm:iabiiiab} are only known for small values of $a$ and $b$.

For $a=b=2$ this theorem was known, because it follows from the following simple fact. If $p=n$ then the contact singularities with algebra $I_{2,2}$ form an open subset of the $\Sigma^2$ germs, while if $p>n$ then the singularities with algebra $III_{2,2}$ form an open subset of the $\Sigma^2$ germs. Hence the theorem reduces to the obvious $\tp_{\Sigma^2}(0)=\tp_{\Sigma^2}(1)^{\flat(2)}$.

The Thom polynomials $\tp_{I_{2,3}}(0)$, $\tp_{III_{2,3}}(1)$ are calculated in \cite{rrtp}, but their relation is not noticed there.

\subsection{Lowering the Thom polynomial of $\Phi_{n,n-1}$} \label{sec:veronese} Recall formula (\ref{eq:phinn-1}), the Thom polynomial of $\Phi_{n,n-1}$:
  \[ \tp_{\Phi_{n,n-1}}(l)=2^{n-1}\sum_{i=0}^{l+1}2^i
       \Delta_{n+l+i,{\underbrace{\scriptstyle{n+l,\dotsc,n+l}}_{\mbox{$\scriptstyle{n-1}$}}},l+1-i}.  \]

 Germs $(\C^n,0) \to (\C^{n+l},0)$ with this algebra only exist with $l\geq \binom{n}{2}$. Yet, choose $m=n+1$, $k=2$ and $l=-1$, and consider the map
$$\phi: C(Y_{\Phi_{n,n-1}}) \to \eta_{\Phi_{n,n-1}}(n,n-1) \subset J^2(n,n-1)$$
of (\ref{phi}). One may check that the image of this map is $\Hom(\sym^2\C^n,\C^{n-1})\subset J^2(n,n-1)$, whose cohomology class is $e(\Hom(\C^n,\C^{n-1})) = \Delta_{{\underbrace{\scriptstyle{n-1,\dotsc,n-1}}_{\mbox{$\scriptstyle{n}$}}}}$. On the other hand, applying $l+1$ times the lowering operator $\flat(n+1)$ to the polynomial $\tp_{\Phi_{n,n-1}}(l)$ we get $2^{n-1}\Delta_{{\underbrace{\scriptstyle{n-1,\dotsc,n-1}}_{\mbox{$\scriptstyle{n}$}}}}$ (using the elementary fact that $\Delta_{(i_1,\dotsc,i_m)}^{\flat(m)}=\Delta_{(i_1-1,\dotsc,i_m-1)}$). Comparing these two cohomology classes implies that the map $\phi$ is a covering with $2^{n-1}$ sheets.

Now we show the classical geometry reason for $\deg \phi = 2^{n-1}$. The definition of $\Phi_{n,n-1}$ implies that $Y_{\Phi_{n,n-1}}\subset \P(\sym^2\C^n)\subset \gr^{n+1}(J^2(n))$ is identified with the projectivization of the set of rank 1 symmetric matrices. This closed variety is the image of the {\em Veronese} map $\P(\C^n) \to \P(\sym^2\C^n)$. Denote the two obvious projections of $C(Y_{{\Phi_{n,n-1}}})$ by $\pi_1$ and $\pi_2$. For a generic $g$ in the image of $\phi$, the set $\pi_1(\pi_2^{-1}(g))$ intersects $\P(\sym^2\C^n)\subset \gr^{n+1}(J^2(n))$ in an $n-1$-codimensional linear subspace, hence the number of $\phi$-preimages of $g$ is the degree of the Veronese variety. This degree is known to be $2^{n-1}$, see e.g. \cite[p.231]{harris:ag}, agreeing with our result above.

\subsection{Thom series for Thom-Boardmann classes}\label{sec:tb} Let $\Sigma^K$ denote the Thom-Boardmann class corresponding to $K=(i_1,\dotsc,i_s)$ for $i_1\geq\cdots\geq i_s\geq 1$ (see e.g. \cite{mathertb}). For $n\geq i_1$ and $p\geq p_0$  ($p_0$ depending on $K$) there is a jet $g_K(n,p)\in J^k(n,p)$, such that $\mathcal{K}^kg_K$ is open in $\Sigma^K$, and hence $\Tp(g_K(n,p))=\Tp_{\Sigma^K}(n,p)$. The nilpotent algebra of any of these $g_K(n,p)$ jets are isomorphic, it will be denoted by $Q_K$. For $n-i_1<p<p_0$ the Thom-Boardmann class $\Sigma^K(n,p)$ is still not empty, but it may split into families of lower dimensional contact classes. The question was raised in \cite{dstab} whether the Thom series of $g_K$ calculates $[\Sigma^K(n,p)]$ for $n-i_1<p<p_0$, too. Consider such a small $p$, a sufficiently large $k$, and the map
$\phi: C(Y_{Q_K}) \to \eta_{Q_K}(n,p)\subset J^k(n,p)$ from (\ref{phi}). The image $\eta_{Q_K}(n,p)$ can be identified with the $\Sigma^K$ germs in $J^k(n,p)$.
The dimensions of the source and target spaces of $\phi$ are the same. Moreover, $\phi$ has a well-known section, $g \mapsto \beta(g):=$the {\em Boardmannization} of $g$ (see \cite[\S 2]{mathertb}). Since the correspondence variety $C(Y_{Q_K})$ is connected, this implies that $\phi$ has degree 1, and we have the following
\begin{theorem}\label{th:tb}The Thom series $\ts_{g_{\raisebox{-.4ex}{$\scriptscriptstyle K$}}}$ calculates the Thom polynomial of $\Sigma^K(n,p)$ for all $n,p$.
\end{theorem}

We can call $\ts_{g_{\raisebox{-.4ex}{$\scriptscriptstyle K$}}}$ the Thom series $\ts(\Sigma^K)$ of $\Sigma^K$. Notice that this way we obtain Thom polynomials for $p<n$ as well, which case is not covered by the Localization Formula. We should mention that at this point this is only a theoretical possibility as $\ts(\Sigma^K)$ is known only in the few cases listed in Section \ref{sec:exa}.

\subsection{Nets of conics} \label{sec:netsofconics} The 1-parameter family of jets
\[g_\lambda=(x^2-\lambda yz,y^2-\lambda xz,z^2-\lambda xy) \ \ \text{for} \ \lambda(\lambda^3-1)(8\lambda^3+1)\not=0.\]
was studied by Mather in \cite{mather5} and Wall in \cite{wall:nets}. This is the smallest codimensional example of a family of jets for $n=p$. The contact class of $g_\lambda$ has codimension 10, their union is open in the Thom-Boardmann class $\Sigma^3(3,3)$. Thom polynomials of contact classes contained in $\Sigma^3(n,n)$ (for any $n$) are linear combinations of $\Delta_{\mu}$  where the Young diagram of the partition $\mu$ contains a $3\times 3$ square (see \cite[\S 4.2]{pragacz:enumgeo}). Therefore $\tp(g_\lambda)=A\Delta_{3331}+B\Delta_{433}$ for some $A,B\in \N$. The restriction equation $\tp(g_\lambda)|_{g_\mu}=0$ implies that $2A=B$.

\begin{theorem}The Thom polynomial of $g_\lambda$ for generic $\lambda$  is
  \begin{equation}\label{eq:nets}  \tp(g_\lambda)=4\Delta_{3331}+8\Delta_{433}.   \end{equation}
\end{theorem}

\noindent{\em Sketch of the proof.} The ideal $I_\lambda$ of $g_\lambda$ in $J^k(3)$, where $k\geq 3$, has depth 3 and we have $\mu(g_\lambda)=7$. Consider the ideal $I'_\lambda=I_\lambda+(J^k(3))^3$, whose depth is 2. This ideal can only be generated by at least 4 polynomials, hence the 6-dimensional nilpotent algebra $Q_\lambda:=J^3(3)/I'_\lambda$ does not corresponds to any germ with $p=n$. Yet, consider $n=p=3$ and the map
$$\phi: C(Y_{Q_\lambda}) \to \eta_{Q_\lambda}(3,3)$$
from (\ref{phi}). The only ideal of codimension 7 in $I'_\lambda$ is $I_\lambda$. This implies that $\phi$ is a birational map to the closure of the contact orbit of $g_\lambda$. We have $Y_{Q_\lambda}\subset X:=j(\gr_3(\Hom(\sym^2\C^3,\C)))$, where $j$ is the obvious embedding $\gr_3(\Hom(\sym^2\C^3,\C))\to \gr^6$ discussed in Section \ref{subgrass} on subgrassmannians.

The action of the right group $\mathcal{R}(3)$ on the 9-dimensional $X$ is studied in \cite{mather5}. It is shown there that the action is equivalent to the action of the 8-dimensional Lie group PGL$(3)$, and the orbit closure $O_\lambda$ of $I^{'}_\lambda$ is 8-dimensional, i.e. a hypersurface.
Ideas of \cite{wall:nets} can be used to show that the degree of $O_\lambda$ for generic $\lambda$ is 4 (and there is one orbit closure with degree 2). Either the Localization Formula, or the idea of {\em projective} Thom polynomials (\cite[\S 6]{forms}) can be used to obtain that this degree is equal to the coefficient of $\Delta_{3331}$ in the the Thom polynomial of $g_\lambda$. \qed

\section{Final remarks}\label{sec:final}

The results and examples of this paper may give the wrong impression that now the Thom polynomials of all singularities are calculated. Although we indeed reached beyond the previously known Thom polynomials (and series), let us demonstrate the boundaries of our knowledge by some open problems.

We do not know how to calculate the Thom series of $A_n$ for $n>6$. For $n>9$ we do not even know the first Thom polynomial $\tp_{A_n}(0)$. We do not know the Thom series of the Thom-Boardmann class $\Sigma^{211}$. Are there closed formulas for classes of singularities, for example $\{A_n: n\geq 0\}$ or $\{I_{a,b}: a,b\geq 2\}$? We repeat a conjecture of the second author \cite{rrtp} in a slightly strengthened form:

\noindent {\bf Conjecture:} Every coefficient of the Thom polynomials $\tp_{A_n}(l)$---written as a linear combination of Chern monomials---is non-negative, and all coefficients of width at most $n$ are strictly positive. (By d-stability, the other coefficients are 0.)

In \cite[Sect. 8.7]{bsz06} this conjecture is verified for $n=3$ and 4.

\bibliography{damon}
\bibliographystyle{alpha}
\end{document}